\documentclass[10pt]{article}
\usepackage[utf8]{inputenc}
\usepackage[T1]{fontenc}
\usepackage{amsmath,amsthm}
\usepackage{amssymb}
\usepackage{color,cleveref}

\newcommand{\tr}{\textrm{\normalfont tr}}
\newcommand{\imm}{\textrm{\normalfont Im}}

\newtheorem{cor}{Corollary}
\newtheorem{lemma}{Lemma}
\newtheorem{prop}{Proposition}
\newtheorem{thm}{Theorem}
\newtheorem{theorem*}{Theorem}

\theoremstyle{remark}
\newtheorem{rmrk}{Remark}
\theoremstyle{definition}
\newtheorem{dfn}{Definition}

\usepackage{xpatch}

\makeatletter
\newcounter{proofpart}
\xpretocmd{\proof}{\setcounter{proofpart}{0}}{}{}
\newcommand{\proofpart}[1]{%
	\par
	\addvspace{\medskipamount}%
	\stepcounter{proofpart}%
	\noindent\emph{Step \theproofpart: #1}\par\nobreak\smallskip
	\@afterheading
}
\makeatother

\makeatletter
\newcounter{claim}
\xpretocmd{\proof}{\setcounter{claim}{0}}{}{}
\newcommand{\claim}[1]{%
	\par
	\addvspace{\medskipamount}%
	\stepcounter{claim}%
	\noindent\emph{Step \theclaim: #1}\par\nobreak\smallskip
	\@afterheading
}
\makeatother

 \title{Functional determinants for the second variation}

\author{Stefano Baranzini \footnote{Università degli studi di Torino,  Via Carlo Alberto 10, Torino, Italy }
	\footnote{email: sbaranzi@sissa.it}}
	
\begin{document}
	\maketitle

		\begin{abstract}
          We study the determinant of the second variation of an optimal control problem for general boundary conditions. Generically, this operators are not trace class and the determinant is defined as a principal value limit. We provide a formula to compute this determinant in terms of the linearisation of the extremal flow. We illustrate the procedure in some special cases, proving some Hill-type formulas.				
		\end{abstract}

	\section{Introduction}
	The main focus of this paper is to study the spectrum of a particular class of Fredholm operators that arise in the context of Optimal Control. Our main result is a formula which relates the determinant of these operators to the fundamental solutions of an ODE system in a finite dimensional space, much in the spirit of Gelfand-Yaglom Theorem.
	
	 For operators of the form $1+K$, where $K$ is a compact operator, various ways of defining a \emph{determinant} function can be found in the literature. Going all the way back to Poincaré, Fredholm and Hilbert. 
	If the operator $K$ one considers is in the so called \emph{trace class}, i.e. the sequence of  its eigenvalues (with multiplicity) gives an absolutely convergent series, a definition of determinant which involves the (infinite) product of its eigenvalues is possible. 

	In our case, however, the classical approach is not immediately applicable since, typically, the operators one encounters are not trace class. Under some technical assumptions, the operators arising as second variation have a symmetric spectrum, as shown in \cite{determinant} and \cite{asymptoticeigenvaluessecondvariation}. Generically, there is a non negative real number $\xi(K)$, which we call \emph{capacity}, for which the following asymptotic for the ordered sequence of the eigenvalues holds:
	\begin{equation}
		\lambda_n(K) = \frac{\xi(K)}{n} + O(n^{5/3}), \quad  n \in \mathbb{Z}.
	\end{equation}
	
	This symmetry allows us to talk about trace and determinant of the operators $K$ and $1+K$ in the sense of \emph{principal value} limits. A similar approach has been independently adopted in the works \cite{conditionalfredholmhu, hillformulahu} to study the spectrum of Hamiltonian differential operators.
	
	There are, of course, many other ways to define a \emph{determinant function} for classes of Fredholm operators. For instance one could apply the theory of regularized determinants, see \cite{simontraceideals}, or rely on the so called $\zeta-$regularization, see
	\cite{bookZeta} for details. The literature concerning these topics is vast. To mention a few works, one could refer to  \cite{KirstenDetContour,kirstenDetSturmLiouville} for some results about Sturm-Liouville problems, to \cite{friedlandergraph} for graphs  and to  \cite{FormanSturmLiouville} results proved in the general framework of elliptic operators on section of vector bundles. A relation between regularized determinants and $\zeta$-regularization is given in  \cite{zetaFredholmLesch}.

	Our approach, even if less general, provides an actual extension of the definition of determinant given for trace class operators. It involves some \emph{principal value} limit of the product of the eigenvalues. Thus, whenever the compact part of the second variation is trace class, it gives exactly the usual Fredholm determinant defined for trace class operators.
	
	It is worth pointing out another feature of the construction. Our formula relates the determinant of $1+K$ to the fundamental solution of a finite dimensional system of (linear) ODEs. This provides a way to actually compute the determinant and allows to recover some classical results such as Hill's formula for periodic trajectories. This kind of formulas have important applications since allow to relate variational properties of an extremal (i.e. the eigenvalues of the second variation), to dynamical properties such as \emph{stability}. These properties are usually expressed through the eigenvalues of the linearisation of the Hamiltonian system of which the extremal we are considering is solution. Take the case of a periodic, non degenerate trajectory. On one hand, knowing the sign of the determinant amounts to know the parity of the Morse index of the extremal. On the other hand this sign is completely determined by the number of positive eigenvalues of the linearisation grater than one. Applications of these kind of ideas go back to Poincaré's result about the instability of closed geodesics and can be found for example in
	\cite{bolotin} or \cite{traceformulahu}. 
	For several interesting examples of the interplay between parity of Morse index and stability see for instance 
	\cite{portaluri_instability_re_plane,portaluri_morse_index_linear_stability} or \cite{portaluri_stability_re_spectral_flow}. Other related works in this direction are \cite{offinHyperbolic,urenaBrake} and \cite{huBrake} . 
	
	We stress that our results are formulated in a quite general framework which encompasses Riemannian, sub-Riemannian, Finsler geometry and mechanical systems on manifolds to name a few. Moreover the techniques can be applied without virtually any modifications to treat constrained variational problems on compact graphs as already done in \cite{morsegraph} to compute the Morse index.

	The structure of the paper is the following: in \Cref{section: statements} we recall the notation and the setting we will use through out the paper and give the full statement of our results. \Cref{section: space of variation} contains some information about the second variation and the structure of the space of variation we will employ. 
	
	In \Cref{section: applications} we deal with a couple of applications, such as Hill's formula and the eigenvalue problem for Schr\"odinger operators. The results deserve some interest on their own, however, the main focus of the section is to provide a worked out example of how to apply the formulas to concrete situations.
	
	The last part of the paper is \Cref{section: proof separate BC,section: proof general BC}. They are devoted to the proof of  \Cref{thm: main theorem}.
	We first prove the result for boundary conditions of the type $N_0 \times N_1$ (\Cref{section: proof separate BC}) and than extend it to the general case (\Cref{section: proof general BC}). Finally we give formulas to compute the trace of the compact part of the second variation $K$ (\Cref{lemma: second variation general BC and trace,lemma: quantities at s=0}) and some normalization constant appearing in \Cref{thm: main theorem}.

\subsection{Problem statement and main results}
\label{section: statements}
We begin this section recalling briefly the setting and the notations that will be used throughout the paper. The reader is referred to \cite{bookcontrol,bookSubriemannian} for more information on optimal control and sub-Riemannian problems. By an optimal control problem we mean the following data: a \emph{configuration space}, i.e. some smooth manifold $M$ and a family of smooth (and complete if $M$ is non compact) vector fields $f_u$. They depend on a parameter $u$ living in some open set $U \subseteq \mathbb{R}^k$. We will always assume that the $f_u$ are smooth jointly in both variables, i.e. they belong to $C^\infty (U\times M,TM)$.  We can think of the parameter $u$ as our way of interacting with the system and moving a particle from one state to another.

To any function $u(\cdot) \in L^{\infty}([0,1],U)$ we can associate a trajectory in the configuration space considering the solution of:
\begin{equation}
	\label{eq:vector_fields}
	\dot q = f_{u(t)}(q), \quad q(0) = q_0 \in M.
\end{equation}
We will usually call the  function $u(\cdot)$ \emph{control}.

We say that a control $u$ is admissible if the corresponding trajectory, denoted by $\gamma_{u}$, is defined on the whole interval $[0,1]$. We can impose further restrictions on the trajectory $\gamma_u$ specifying proper boundary conditions. The most general situation that we are going to treat in the paper, is the case in which the boundary conditions are given by a submanifold $N \subseteq M \times M$. We say that (a Lipschitz continuous) $\gamma$ is admissible if $\gamma = \gamma_u$ for some admissible control $u$ and if $(\gamma(0),\gamma(1)) \in N$.

Given a smooth function $\varphi : U \times M \to \mathbb{R}$, we are interested in the following minimization problem on the space of \emph{admissible} curves:
\begin{equation}
	\label{eq: minimaztion problem}
	\min_{\gamma_u\text{ admissible}} \mathcal{J}(\gamma_u) = \min_{\gamma_u \text{ admissible}} \int_0^1 \varphi(u(t),\gamma_{u}(t))dt
\end{equation}

It is customary to parametrized the space of admissible curves using the \emph{velocity}, i.e. the control function $u(\cdot)$, and a finite dimensional parameter space which takes into account initial data and boundary conditions. We are going to follow this approach. However, this is just a technical point which is independent of the main statements. Hence we postpone the discussion of the structure of our space of variations to \Cref{section: space of variation}. Let us just mention that, under some natural assumptions, the space of variations can be endowed, locally, with a smooth Banach manifold structure. Thus, it makes sense to consider the tangent space to the space of variations. It is a finite codimension subspace $\mathcal V$ of  $L^\infty ([0,1],\mathbb{R}^k) \oplus\mathbb{R}^{\dim(N)}$. Suppose that $u$  is  critical point of the functional $\mathcal{J}$ restricted to our space of variations and consider the Hessian of $\mathcal{J}$. It is a quadratic form on $\mathcal{V}$. We denote this quadratic form by $Q(v) = d_u^2\mathcal{J}\vert_\mathcal{V}(v,v)$.
Instead of working with $L^\infty$ topology, we will work with the weaker $L^2$ one since everything extends by continuity. For an appropriate choice of scalar product on $L^2([0,1],\mathbb{R}^k)$, it turns out that the quadratic form $(Q-I)(v) = Q(v)-\langle v,v\rangle$ is compact, but in general not trace-class. For a detailed account on the spectrum of the second variation the reader is referred to \cite{determinant, asymptoticeigenvaluessecondvariation}. 
Given an eigenvalue of $Q$, $\lambda$, denote by $m(\lambda)$ its multiplicity. We define the determinant of the second variation as the following limit:
\begin{equation*}
	\det( Q) = \lim_{\epsilon \to 0} \prod_{\vert \lambda-1\vert>\epsilon} \lambda^{m(\lambda)}, \text{ where } \lambda \in \text{Spec}(Q).
\end{equation*}

As already stated in the introduction, the computation of this determinant for general boundary conditions is the main contribution of this work. We provide a formula for this determinant involving essentially two ingredients:
\begin{itemize}
\item the fundamental solution of a linear (non autonomous) system of ODE's which we call \emph{Jacobi} equation;
\item the annihilator to the boundary condition manifold, a Lagrangian submanifold of $T^*M$.
\end{itemize}

To state the main Theorem and define precisely the objects above, we need to introduce a little bit of notation. We will just sketch here what is needed to this purpose, further details are collected in \ref{appendix} or given along the proofs. From now on, assume that a strictly normal extremal $\lambda_t$ with optimal control $\tilde u$ is fixed (see \ref{appendix}).

The first tool we introduce is the following family of Hamiltonians. It is strictly related to Legendre transform and quite useful when dealing with problems in the cotangent bundle. Given our optimal control $\tilde{u}(\cdot)$ define:
\begin{equation*}
	h^t_{\tilde u} (\lambda)= \langle \lambda, f_{\tilde{u}(t)}(q) \rangle -\varphi(\tilde{u}(t), q), \quad \dot{\lambda} = \vec{h}^t_{\tilde{u}}(\lambda).
\end{equation*}
Denote its flow at time $t$ by $\tilde \Phi_t$ and its differential by $(\tilde \Phi_t)_*$. Pontryagin Maximum Principle (see \ref{appendix})  tells us that normal extremals $\lambda_t$ satisfy $\lambda_t = \tilde \Phi_t(\lambda_0)$. We will use the map $\tilde \Phi_t$ to connect the tangent spaces to each point of $\lambda_t$ to the starting one, $\lambda_0$. This flow, in some sense, plays the role of the choice of a connection (or parallel transport as in \cite{bolotin}).

The second object we are going to introduce, is a kind of quadratic approximation of our starting system. It is given by a quadratic Hamiltonian on $T_{\lambda_0} T^*M$ (see for detail \cite{ASZ} or \cite{bookcontrol}[Chapters 20 and 21]).
To define it we need to introduce two matrix valued functions $Z_t \in \mathrm{Mat}_{k\times 2\dim(M) }(\mathbb{R})$ and $H_t \in \mathrm{Mat}_{k\times k}(\mathbb{R})$.  The precise way to compute them is given in \eqref{eq: Z_t and H_t}. However, for the moment, a precise understanding of how this matrices are obtained is not strictly needed. Heuristically, the matrix $Z_t$ represents a linear approximation of the Endpoint map of the original system whereas $H_t$ is a quadratic approximation of the Lagrangian $\varphi$ along the extremal.

 Let  $\pi : T^*M \to M$ be the natural projection and set $\Pi:=\ker \pi_*$, the \emph{fibres}. Define $\delta^s$ as the dilation by $s \in \mathbb{R}$ of $\Pi$. It is determined by the relations $\pi_*(\delta^s w) = \pi_*(w)$ for all $w \in T_{\lambda_0}T^*M$ and $\delta^s v = s v$ for all $v \in \ker (\pi_*)$. Let $J$ be some coordinates representation of the standard symplectic form on $T_{\lambda_0} T^*M$. 
 Let us define the following quadratic form:
\begin{equation*}
	\beta_s(\lambda) = \frac{1}{2}\langle \lambda, J \delta^sZ_t H_t^{-1}(\delta^sZ_t)^*J \lambda\rangle, \quad \lambda \in T_{\lambda_0}T^*M.
\end{equation*}
We will call \emph{Jacobi (or Jacobi type) equation} the following ODEs system on $T_{\lambda_0} T^*M$:
\begin{equation*}
	\dot \lambda = \vec{\beta}_s(\lambda),  \quad \lambda \in T_{\lambda_0}T^*M.
\end{equation*}
and denote its fundamental solution at time $t$ as $\Phi_t^s$. Here, and for the rest of the paper, we will call \emph{fundamental solution} any family $\Theta_t$, $t \in \mathbb{R}$ of linear maps which satisfies a linear ODE and have initial  condition $\Theta_0 = I$.
\begin{rmrk}
	Whenever PMP's maximum condition determines a $C^2$ function $h$ on $T^*M$, normal extremals satisfy a Hamiltonian ODE on the cotangent bundle of the form $\dot \lambda = \vec{h}(\lambda)$.  Jacobi equation for $s=1$ is closely related to the linearisation of $\vec{h}$ along the extremal we are fixing. 
	Suppose local coordinates are fixed and let $d^2_{\lambda_t}h$ be the Hessian of $h$ along the extremal. Let $\Psi_t$ be the fundamental solution of:
	\begin{equation*}
		\dot\Psi_t = d^2_{\lambda_t}h \, \Psi_t.
	\end{equation*} 
	It can be shown (see for example \cite{bookcontrol}[Proposition 21.3]) that:
	$$\Psi_t = (\tilde \Phi_t)_*\Phi_t^1.$$
\end{rmrk}

The last maps we will need are a family of symplectomorphism of $T_{\lambda_0} T^*M $ and $ T_{\lambda_1} T^*M $. Their definition depends on the choice of a scalar product on each tangent space. Let $g_0$ and $g_1$ be two scalar products on these spaces. Assume that at each $\lambda_i$, $\Pi_i := \ker (\pi_*) \subseteq T_{\lambda_i}T^*M$ has a Lagrangian orthogonal complement with respect to $g_i$ which we denote by $\Pi_i^\perp$. For a subspace $V$, denote by $pr_V$ the orthogonal projection onto $V$. We set:
\begin{equation}
	\label{eq:def_Ai_intro}
	\begin{split}
		A_0^s (\eta) &= \eta + (1-s) J_0^{-1}pr_{\Pi_0^\perp} \eta, \quad  \eta \in T_{\lambda_0}(T^*M), \\
		A_1^s(\eta) &= \eta+ (1-s)(J_1^{-1} +\tilde \Phi_*\circ pr_{\Pi_0}\circ\tilde \Phi_*^{-1})pr_{\Pi_1^\perp} \eta, \quad  \eta \in T_{\lambda_1}(T^*M).
	\end{split}
\end{equation} 	

The datum of the boundary condition is encoded in a Lagrangian submanifold of $\big (T^*(M \times M),(-\sigma)\oplus \sigma \big)$, the annihilator of $N$. It can be thought of as the symplectic version of the normal bundle in Riemannian geometry and is define as follows. Take a sub-manifold $N \subseteq M \times M$ and consider:
\begin{equation*}
	Ann(N) = \bigcup_{q \in N}\{(\lambda_0,\lambda_1) \in T_q^*(M\times M) : \langle \lambda_0, X_0 \rangle  = \langle \lambda_1,X_1\rangle, \forall (X_0,X_1) \in T_q N\}.
\end{equation*} 
In light of PMP (see \ref{appendix}), critical points of $\mathcal J$ with boundary conditions given by $N$, lift to the cotangent bundle to curves $\lambda_t$ such that $(\lambda_0,\lambda_1) \in Ann(N)$. 

Fix now a complement to $T_{(\lambda_0,\lambda_1)}Ann(N)$, say $V_N$, and denote by $\pi_N$ the projection on $V_N$ having $T_{(\lambda_0,\lambda_1)}Ann(N)$ as kernel.
We are ready now to define a function that plays the role of the \emph{characteristic polynomial} of the Hessian of $\mathcal{J}$. For a map $f$ denote by $\Gamma(f)$ its graph, set:
\begin{equation*}
	\mathfrak{p}_Q(s) = \det (\pi_N\vert_{\Gamma(A_1^s\tilde \Phi_* \Phi^s_1A_0^s)}).	
\end{equation*}
 
With this notation, our main result reads as follows:
\begin{thm}
	\label{thm: main theorem}
	Suppose that $\lambda_t$ is a strictly normal extremal for problem \eqref{eq: minimaztion problem} and $\tilde{u}$ is its optimal control. 
	Moreover, suppose that $\lambda_t$ satisfies Legendre strong condition, i.e. that  $ \exists \alpha >0$ such that, $ \forall v \in \mathbb{R}^k$ 
	$$ \langle -H_t v, v\rangle \ge\alpha \langle v, v\rangle $$
	and that at least one of the following holds:
	\begin{itemize}
		\item  the maps $t \mapsto Z_t$ and $t \mapsto H_t$ are piecewise analytic in $t$;
		\item the dimension of the space of controls is $k \le 2$;
		\item the operator $I-Q$ is trace class;
	\end{itemize}  
	Let $\lambda \in \mathrm{Spec}(Q)$ and $m(\lambda)$ be its multiplicity, the limits:
	\begin{equation*}
		\det(Q) = \lim_{\epsilon \to 0}\prod_{\vert\lambda-1\vert >\epsilon} \lambda^{m(\lambda)},\quad \tr(Q-I) =\lim_{\epsilon \to 0} \sum_{\vert\lambda-1\vert >\epsilon}m(\lambda)(\lambda-1).
	\end{equation*}
	are well defined and finite. Moreover, for almost any choice of metrics $g_0$ and $g_1$ we have that $\mathfrak{p}_Q(0) \ne 0$ and that:
	\begin{equation*}
		\det(I+s(Q-I)) = \mathfrak{p}_Q(0)^{-1}e^{s(\tr(Q-I)-\mathfrak{p}_Q'(0) )}\mathfrak{p}_Q(s).
	\end{equation*} 
\end{thm}

\begin{rmrk}
	The hypothesis about the regularity of $Z_t$ and $H_t$ are needed to obtain the asymptotic for the spectrum of $Q-I$ that guarantees the existence of the trace and of the determinant as limits. They can be weakened somehow by requiring that the skew-symmetric $k \times k $ matrix $Z^*_t J Z_t$ is continuously diagonalizable (see \cite{determinant}).
\end{rmrk}
\begin{rmrk}
	The constants $\mathfrak{p}_Q(0)$, $\mathfrak{p}_Q'(0)$ and $\tr(Q-I)$ are completely explicit and are given in terms of iterated integrals in \Cref{lemma: quantities at s=0,lemma: second variation general BC and trace}. 
\end{rmrk}

In particular we have the following corollary:
\begin{cor}
	Under the assumption above, the determinant of the second variation $Q$ satisfies:
	\begin{equation*}
			\det(Q) = \mathfrak{p}_Q(0)^{-1}e^{\tr(Q-I)-\mathfrak{p}_Q'(0) }\det(\pi_N \vert_{\Gamma(\Psi_t)}).
	\end{equation*}
	Where $\Psi_t = \tilde \Phi_* \Phi_1$ and coincides with the fundamental solution of the linearisation of the extremal flow, whenever the latter is defined.
\end{cor}

	\section{The Second Variation}
\label{section: space of variation}
The aim of this section is threefold: to define precisely what we mean by $d^2\mathcal{J}\vert_\mathcal{V}$, to define precisely its domain and to provide the integral representation of this quadratic form we will use throughout the proof section of the paper.

Before going on, a little remark about topology is in order. Up until now we have considered Lipschitz continuous curves and $L^\infty$ controls. Hence, it would be natural to work on the Banach space $L^{\infty}([0,1],\mathbb{R}^k) \oplus \mathbb{R}^{\dim(N)}$. However it turns out that, even if $d^2 \mathcal{J}\vert_\mathcal{V}$ is defined on the latter space, it extends to a continuous quadratic form on $L^2([0,1],\mathbb{R}^k)\oplus\mathbb{R}^{\dim(N)}$. Moreover, critical points of $d^2 \mathcal{J}\vert_\mathcal{V}$ in $L^2([0,1],\mathbb{R}^k)$ are continuous and thus belong to $L^\infty([0,1],\mathbb{R}^k)$. For this reason (and Fredholm alternative), we will work with $L^2$ controls for the rest of the paper.

Let $n_0,n_1 \in \mathbb{N}$ and consider the Hilbert space $\mathcal H = \mathbb{R}^{n_0} \oplus L^2([0,1],\mathbb{R}^k)\oplus \mathbb{R}^{n_1}$ (its scalar product will be defined in the next section). 
Let $(\Sigma,\sigma)$ be a symplectic space and consider a linear map $Z : \mathcal H \to \Sigma$ defined as:
\begin{equation*}
	Z(u) = Z_0u_0+\int_0^1Z_t u_t dt +Z_1u_1, \quad u = (u_0,u_t,u_1) \in \mathcal{H}.
\end{equation*}
Suppose that  $\Pi\subset \Sigma$ is a Lagrangian subspace transverse to the image of the map $Z$ and define $\mathcal V = Z^{-1}(\Pi)$. For an appropriate choice of $Z$ and $\Pi$ which depends on \eqref{eq:vector_fields} and \eqref{eq: minimaztion problem}, the \emph{second variation} (at a strictly normal critical point) is the quadratic form given in the following definition.
\begin{dfn}[Second Variation]
	\label{def second variation}
	The second variation at $\tilde{u}$ is the quadratic form defined on  $\mathcal{V} \subseteq \mathcal{H}$:
	\begin{equation}
	\label{eq: second variation quadratic form}
	\begin{aligned}
	Q(u) = \int_0^1\left ( -\langle  H_t u_t,u_t\rangle + \sigma(Z_t u_t ,\int_0^t Z_\tau u_\tau d \tau +Z_0u_0) \right )  dt + \\ + \sigma(Z_0u_0+\int_0^1 Z_t u_t dt,Z_1u_1).
\end{aligned}
\end{equation}
\end{dfn}

The definition of this quadratic form may be a bit strange at first glance. Despite the appearances, the way one gets to such an expression is quite natural. The construction is explained in detail in \cite{morsegraph}. We will sketch here just the main features, essentially to introduce the notation needed. 

The idea is to reduce the problem with boundary conditions $N$ to a fixed points (or Dirichlet) problem for an appropriate auxiliary system. We will consider just the case of separated boundary conditions $N_0 \times N_1$. The general case reduces to this one using the procedure explained in \Cref{section: proof general BC}.

The first step of the construction is to build the auxiliary system. We always work with a fixed strictly normal extremal $\lambda_t$. Fix local foliations in neighbourhoods of its initial and final points having  a portion of $N_0$ and $N_1$ as leaves. This determines two integrable distributions in a neighbourhood of those points. Suppose that said distributions are generated by some fields $\{X_i^j\}_{i=1}^{\dim(N_j)}$ and $j=0,1$. Consider the extended system:
\begin{equation*}
	f_u^t(q) = \begin{cases}
		\sum_{i=0}^{\dim(N_0)} X_i^0(q)u_0^i, &\text{ for } t \in [-1,0)\\ f_u(q), &\text{ for } t \in [0,1]\\ 	\sum_{i=0}^{\dim(N_1)} X_i^1(q)u_1^i, &\text{ for } t \in (1,2].
	\end{cases}
\end{equation*}
Denote the initial and final points of our original extremal curve by $(q_0,q_1)$. We will use controls that are locally constant on $[0,1]^c$, this will be enough to reach any neighbouring point of $(q_0,q_1)$ in $N_0\times N_1$. Minimizing our original functional is equivalent to minimize, with Dirichlet boundary conditions, the following one:
\begin{equation*}
	\mathcal{J}(\gamma_{(u_0,u,u_1)})= \int_0^1 \varphi(u(t),\gamma_{(u_0,u,u_1)} (t) )dt.
\end{equation*}

The second step is to differentiate the Endpoint map (see \ref{appendix}) of the auxiliary system. We employ the machinery of Chronological Calculus (see also \cite{bookcontrol}[Section 20.3]), which is standard for fixed endpoints. One of the main steps of this differentiation, is to use a suitable family of symplectomorphism to trivialize the cotangent bundle along the curve we are fixing. This allows to write all the equations in the tangent space to the initial point, $T_{\lambda_0}(T^*M)$. 
Let us consider the following functions depending on the parameter $u$:
\begin{equation*}
	h^t_u(\lambda) = \langle \lambda, f_u^t(q) \rangle -\varphi_t(u,q), \quad q = \pi(\lambda).
\end{equation*} 
When an optimal control $\tilde u(t)$ is given, we consider $h^t_{\tilde{u}(t)}(\lambda)$ and the Hamiltonian system:
\begin{equation*}
	\dot{\tilde\Phi}_t = \vec{h}^t_{\tilde{u}(t)}(\tilde \Phi_t).
\end{equation*}
We then define the following functions:
\begin{equation*}
	\begin{split}
	Z_t &= \partial_u \vec h_u^t(\tilde \Phi_t(\lambda))\vert_{\lambda = \lambda_0}, \quad  Z_t : \mathbb{R}^k \to T_{\lambda_0}(T^*M), \\ H_t &= \partial_u^2 h_u^t(\tilde \Phi_t(\lambda))\vert_{\lambda = \lambda_0},\quad  H_t: \mathbb{R}^k \to \mathbb{R}^k.
\end{split}
\end{equation*}

The asymptotic expansions of Chronological Calculus tell us that the second variation at $\tilde u$ is the following quadratic form:
	\begin{equation*}
		\label{eq: scnd variation quadratic form}
		Q(u) = \int_{-1}^2\langle - H_t u,u\rangle - \int_{-1}^2\int_{-1}^t \sigma (Z_\tau u_\tau, Z_t u_t) d\tau dt.
	\end{equation*}
It is defined on the tangent space $\mathcal{V}$ to the variations fixing the endpoints of our curve. This space can be described explicitly as $\mathcal{V } = \{v : \int_{-1}^2 Z_t v_t dt \in \Pi\}$, where $\Pi = \ker \pi_*$.

The third step is to specialize this representation to our auxiliary system. Notice that an extremal of the original problem lifts naturally to an extremal of the auxiliary one. If $\tilde u$ is the original optimal control, extending it by zero on $[0,1]^c$ gives the optimal control for the auxiliary problem.
Applying the construction just sketched to the extended system, we find that $Z_t$ and $H_t$ are locally constant on $[0,1]^c$. We denote $Z_0$ to be its value on $[-1,0)$ and $Z_1$ the value on $(1,2]$. $H_t$ is zero outside $[0,1]$. Thus, after substitution, we recover precisely the operator given in \eqref{eq: second variation quadratic form}.

\begin{rmrk}
We always assume that our extremal is \emph{strictly normal} and satisfies \emph{Legendre strong  condition}. In terms of the matrices $Z_t$ and $H_t$ this means that for $t \in[0,1]$:
	\begin{equation*}
		X_t := \pi_*Z_t\,  \text{ satisfies }\, \int_0^1 X_t^*X_t>0, \quad  \langle -H_t v,v\rangle \ge\alpha \langle v,v \rangle>0, \, \forall v \in \mathbb{R}^k.
	\end{equation*}
\end{rmrk}

As a last remark, notice that, by the first order optimality conditions, the map $Z_0$ takes values in the space $T_{\lambda_0}Ann(N_0)$ and the map $Z_1$ in the space $(\tilde\Phi^{-1}_1)_*\left (T_{\lambda_1}Ann(N_1)\right )$ ( see PMP in \ref{appendix} and \cite{morsegraph}).

\subsection{The scalar product on the space of variations}

As already mentioned, we will assume through out this paper \emph{Legendre strong condition}. The matrix $-H_t$ is positive definite on $[0,1]$, with uniformly bounded inverse. This allows to use $-H_t$ to define an Hilbert structure on $L^2([0,1],\mathbb{R}^k)$ equivalent to the standard one.
We have still to define the scalar product on a subspace transversal to $\mathcal{V}_0 = \{u_0 = u_1=0\}$. A natural choice would be to introduce two metrics on $T_{\lambda_0} T^*M$ and $T_{\lambda_1}T^*M$ and pull them back to the space of controls using the maps $Z_0$ and $\tilde \Phi_*Z_1: \mathbb{R}^n \to T_{\lambda_i} T^*M$. Let us call any such metrics $g_0$ and $g_1$.
\begin{dfn}
	\label{def:scalar_product_on_H}
	For any $u,v \in \mathcal{H}$ 	 define:
	\begin{equation*}
		\langle u, v \rangle = -\int_0^1 H_t(u_t,v_t) dt +   g_0 (Z_0 u_0,Z_0v_0 )  + g_1( \tilde \Phi_*Z_1 u_1,\tilde \Phi_*Z_1v_1 ) 
	\end{equation*}
\end{dfn}

Since the symplectic form $\sigma$ is  a skew-symmetric bilinear form, there exists a $g_i-$skew-symmetric linear operator $J_i$ such that:
\begin{equation*}
	g_i(J_iX_1,X_2) = \sigma(X_1,X_2), \quad  \forall X_1,X_2 \in T_{\lambda_i} T^*M, \, i=0,1.
\end{equation*}


In terms of the symplectic form the scalar product can be written as:
\begin{equation*}
	\langle u, v\rangle =-\int_0^1 H_t(u_t,v_t) dt + \sigma (J_0^{-1}Z_0u_0,Z_0u_0)+\sigma (J_1^{-1}\tilde \Phi_*Z_1u_1,\tilde \Phi_*Z_1u_1)
\end{equation*}

Now, we use the Hilbert structure just introduced to write the operator $K$ associated to the quadratic form $Q-I$, which is compact.
To simplify notation, we can perform the change of coordinates in $L^2$ sending $v_t \mapsto (-H_t)^{\frac{1}{2}} v_t$ and substitute $Z_t$ with $Z_t (-H_t)^{-\frac{1}{2}}$. In this way the Hilbert structure on the interval becomes the standard one. 

We introduce a further piece of notation, call $pr_0$ (respectively $pr_1$) the orthogonal projection on $\imm(Z_0)$ (respectively $\imm(\tilde \Phi_*Z_1)$) with respect to scalar product $g_0$ (respectively $g_1$). Let $L$ be a partial inverse to $\tilde \Phi_*Z_1$ i.e. a map $L : T_{\lambda_0} T^*M \to \mathbb{R}^n$  defined by the relation $L \tilde \Phi_* Z_1 v_1 = v_1$. Set:
\begin{equation}
	\label{eq: lambda u}
	\Lambda(u) =  L  \,pr_1 \,J_1\tilde \Phi_*\Big(Z_0u_0 + \int_0^1 Z_tu_t dt+Z_1u_1\Big) 
\end{equation}

\begin{lemma}
	\label{lemma: compact part second variation}
	The second variation, as a bilinear form, can be expressed as:
	$Q(u,v) = \langle u+Ku,v\rangle$ where $u, v \in \mathcal{V}$ and $K$ is the operator defined by:
	\begin{equation}
		\label{eq: K}
		Ku  = \begin{pmatrix}
			-u_0 \\ - Z_t^*J \Big(\int_{0}^t Z_\tau u_\tau d\tau +Z_0u_0\Big ) \\ -u_1- \Lambda(u)	
		\end{pmatrix}
	\end{equation}
	where $\Lambda(u)$ is given above, in \cref{eq: lambda u}.
	\begin{proof}
		A quick manipulation of the expression involving the symplectic form in \Cref{def second variation}, yields the following:
		\begin{equation*}
			\begin{aligned}
				\int_{-1}^2\int_{-1}^t \sigma (Z_\tau u_\tau, Z_t v_t) d\tau dt = \int_{-1}^0 \int_{-1}^t \sigma (Z_\tau u_\tau, Z_t v_t) d\tau dt   + \int_0^1 \int_{-1}^t \sigma (Z_\tau u_\tau, Z_t v_t) d\tau dt  \\+\int_1^2 \int_{-1}^t \sigma (Z_\tau u_\tau, Z_t v_t) d\tau dt  \\
				=  \int_0^1 \sigma \left(\int_{0}^t Z_\tau u_\tau d\tau +Z_0u_0,Z_t v_t \right) dt +\sigma \Big(Z_0u_0 + \int_0^1 Z_tu_t dt+Z_1u_1,Z_1v_1\Big) \\
				= \int_0^1 \sigma \left(\int_{0}^t Z_\tau u_\tau d\tau +Z_0u_0,Z_t v_t\right) dt +g_1 \Big(J_1 \tilde \Phi_*\Big(Z_0u_0 + \int_0^1 Z_tu_t dt+Z_1u_1\Big),\tilde \Phi_*Z_1v_1\Big) 
			\end{aligned} 
		\end{equation*}
		Recall that $Z_t$ is constant on $[0,1]^c$. Moreover the images of the maps $Z_0$ and $Z_1$ are isotropic subspaces. We used this fact to simplify the expression in the first line.
		Now, it is clear that in the last term:
		$$g_1 \Big(J_1\tilde \Phi_*\Big(Z_0u_0 + \int_0^1 Z_tu_t dt+Z_1u_1\Big),\tilde \Phi_*Z_1v_1\Big),$$
		only the projection onto the image of $\tilde \Phi_*Z_1$ plays a role. It is straightforward to check that:
		$$ g_1 \Big(J_1\tilde \Phi_*\Big(Z_0u_0 + \int_0^1 Z_tu_t dt+Z_1u_1\Big),\tilde \Phi_*Z_1v_1\Big)  = g_1(\tilde \Phi_*Z_1 \Lambda(u),\tilde \Phi_*Z_1v_1).$$

		Recall that we have normalized $H_t$ to $-1$, thus the first summand can be rewritten as follows:	
		\begin{equation*}
			\int_0^1 \sigma \left(\int_{0}^t Z_\tau u_\tau d\tau +Z_0u_0,Z_t v_t\right) dt = \int_0^1 \langle  Z_t^*J \Big(\int_{0}^t Z_\tau u_\tau d\tau +Z_0u_0\Big ), v_t\rangle dt
		\end{equation*}
		
		Adding and subtracting $g(Z_0u_0,Z_0v_0)$ and $g(Z_1u_1,Z_1v_1)$ to single out the identity, we obtain the formula in the statement.
	\end{proof}
\end{lemma}

\section{Hill-type formulas}
\label{section: applications}
Before going to the proof of \Cref{thm: main theorem} we present here some applications of the main result. We deduce Hill's formula for periodic trajectory and specify it to the eigenvalue problem for Schr\"odinger operators. In the second sub-section we present a variation of the classical Hill formula for systems with drift. We will mainly deal with periodic and quasi-periodic boundary conditions. Namely, we consider the case $N = \Gamma(f)$ for a diffeomorphism of the state space $f: M \to M$.

The proofs of this section rely quite heavily on the machinery introduce in \Cref{section: proof general BC}, in particular in \Cref{lemma: quantities at s=0,lemma: second variation general BC and trace}.  The statements, on the other hand, do not and could shed some light on \Cref{thm: main theorem}.
Despite the appearance, proofs are rather simple. They reduce to a (long) computation of the normalizing factors appearing in the statement of \Cref{thm: main theorem} and can be skipped at first reading.
\subsection{Driftless systems and classical Hill's formula}
In this section we consider driftless systems with periodic boundary conditions on $\mathbb{R}^{n}$ and specify the formulas of \Cref{thm: determinant general BC} for this class of problems.

First of all let us explain what we mean by \emph{driftless} systems. Let $t \mapsto R_t$ a continuous family of symmetric matrices of size $n \times n$ and let us denote by $u$ a function in $L^\infty([0,1],\mathbb{R}^n)$.

Consider the following family of vector fields $f_u(q)$, their associated trajectories $q_u(t)$ and the action functional $\mathcal{A}(u)$:
\begin{equation}
	\label{eq: functional hill formula}
	\begin{split}
		f_u(q) = u(t), \quad \begin{cases}
			\dot q_u = f_u(q) = u(t),\\  
			q(0) = q_0 \in \mathbb{R}^n
		\end{cases}\\  \mathcal{A}(u) = \frac{1}{2} \int_0^1 |u|^2-\langle R_t q_u(t),q_u(t) \rangle dt.
	\end{split}
\end{equation}

We impose periodic boundary conditions, i.e. we take $N = \Delta = \{(q,q)\in \mathbb{R}^{2n}:  q \in \mathbb{R}^n\}$. The Hamiltonian coming from the Maximum Principle takes the form:
\begin{equation}
	\label{eq: hamiltonian hill}
	H (p,q) = \max_{u \in \mathbb{R}^k} \langle p,u \rangle -\frac{1}{2}(\vert u\vert^2 - \langle R_t q,q\rangle) = \frac{1}{2}(\langle p,p \rangle+\langle R_t q,q \rangle ).
\end{equation}

Let us denote the flow generated by $H$ by $\Psi_t$. Fix a normal extremal $\lambda_t$ for periodic boundary conditions and its optimal control $\tilde u (t)$. 
\begin{thm}[Hill's formula]
	\label{thm hill formula}
	Let $I+K$ be the second variation at $\tilde{u}$ and $\Psi$ the fundamental solution of $\dot{\Psi} = \vec{H} \Psi$. The following equality holds:
	\begin{equation*}
		\det(I+K) = (-1)^n (2e)^{-n}\det(G^{-1})\det(1-\Psi)
	\end{equation*}
	where $G$ is a scalar product on the tangent space to the initial point.
\end{thm}
\begin{rmrk}
	If we are working on the interval $[0,T]$ instead of $[0,1]$ everything remains essentially unchanged. The only difference is that extra factor $T^{-n}$ appears in the right hand side. In the notation of the proof below this corresponds to $\det(\Gamma)^{-1}$.
\end{rmrk}
We can apply the previous result to study boundary value problems for Sturm-Liouville operators. Let us illustrate the case of Schr\"odinger  equation with periodic boundary conditions. Fix, without loss of generality, the normal extremal  $(p(t),q(t)) = (0,0)$ and with relative optimal control $\tilde{u} = 0$. Consider the cost $\tilde R_t = R_t+\lambda$, for $\lambda \in \mathbb{R}$. Consider the second variation of the functional
\begin{equation*}
	\mathcal{A}_\lambda (u) = \frac{1}{2}\int_0^1 \vert u_t\vert ^2+\langle (R_t+\lambda) q_u(t),q_u(t)\rangle dt
\end{equation*}
at the point $\tilde u =0$. It is given by the operator $1+K_\lambda$ where:
\begin{equation*}
	\langle K_\lambda (u),u \rangle =\lambda\left ( \int_0^1\int_0^t\langle  (\tau-t)u(\tau), u(t)\rangle d\tau dt +\langle u_0,u_0\rangle \right )+\langle K_0(u),u\rangle
\end{equation*}We have the following corollary:
\begin{cor}
	Let $\lambda \in \mathbb{R}$, $\Psi_\lambda$ the fundamental matrix of the lift to $\mathbb{R}^{2n}$ of the following $ODE$ on $\mathbb{R}^{2n}$:
	\begin{equation*}
		\ddot{q}(t) = (R_t+\lambda)q(t)
	\end{equation*}
	The determinant of the operators $1+K_\lambda$ can be expressed as:
	\begin{equation*}
		\det(1+K_\lambda) = (-1)^n (2e)^{-n} \det(G^{-1})\det(1-\Psi_\lambda),
	\end{equation*}
	where $G$ as in the previous statement.
\end{cor}

\begin{proof}[Proof of \Cref{thm hill formula}]
We are going now to describe explicitly all the objects involved in \Cref{thm: main theorem}. Let us start with the flow $\tilde \Phi$ we use to re-parametrize the space  and its differential. It is given by the Hamiltonian:
\begin{equation*}
	\begin{split}
		&h_{\tilde{u}(t)}(p,q) = \langle p,\tilde{u}(t)\rangle+ \frac{1}{2} \langle R_t q, q \rangle \quad  \Rightarrow  \quad \begin{cases}
			\dot{p} = -R_t q \\
			\dot{q} = \tilde u (t).
		\end{cases}\\
		&\tilde \Phi_t(p,q) = \begin{pmatrix}
			1 &-\int_0^t R_\tau d\tau\\
			0 &1
		\end{pmatrix}\begin{pmatrix}
			p\\ q
		\end{pmatrix} +\begin{pmatrix}-
			\int_0^t \int_0^\tau R_\tau  \tilde u(r) dr d\tau\\ \int_0^t  \tilde u (\tau) d\tau 
		\end{pmatrix}\\  &(\tilde \Phi_t)_*  = \begin{pmatrix}
			1 &-\int_0^t R_\tau d\tau\\
			0 &1
		\end{pmatrix}.
	\end{split} 
\end{equation*}
The matrix $Z_t$ is the following:
\begin{equation*}
	\quad Z_t = (\tilde{\Phi}_t^{-1})_* \partial_u \vec h^t_u = \begin{pmatrix}
		\int_0^t R_\tau d\tau \\ 1
	\end{pmatrix}.
\end{equation*}

To simplify notation, let us call $\hat R = -\int_0^1 R_\tau d \tau$.
The annihilator of the diagonal is simply the graph of the identity. Hence the following map, defined on $(T_{\lambda_0} T^*M)^2$, has the latter as kernel:
\begin{equation*}
	(\eta_0,\eta_1) \mapsto \eta_1-\eta_0.
\end{equation*}

We will now define $Q^s$ as in \cref{eq: def map Q} (actually up to a scalar, but this is irrelevant). For $\eta \in T_{\lambda_0} T^*M$ set:
\begin{equation*}
	Q^s(\eta) = \begin{pmatrix}
		-1 &1
	\end{pmatrix} \begin{pmatrix}
	\eta \\ A_1^s\tilde \Phi_* \Phi_1^s A_0^s\eta
\end{pmatrix} =( A_1^s\tilde \Phi_* \Phi_1^sA_0^s -1)\eta \in T_{\lambda_0} T^*M.
\end{equation*}
It is clear that the kernel of $Q^s$ is precisely the intersection of the graph with the diagonal subspace. Since we are working on $\mathbb{R}^{2n}$ we can define the determinant of this map as:
\begin{equation*}
	\det(Q^s) =  \det(A_1^s\tilde \Phi_* \Phi_1^sA_0^s -1)
\end{equation*}
As already mentioned in \Cref{section: statements} this function is a multiple of the \emph{characteristic polynomial} of $K$. It satisfies (see \cref{section: proof separate BC,section: proof general BC}):
\begin{equation*}
	\det(Q^s) = a e^{b s}\det(1+sK), \quad a \in \mathbb{C}^*, b \in \mathbb{C} 
\end{equation*}
Let us compute the normalization factors. To do so, we have to evaluate $\det(Q^s)$ and its derivative at $s=0$. This will give us the relations:
\begin{equation*}
	\det(Q^0) = a, \quad a(b+\tr(K)) = \partial_s \det(Q^s)\vert_{s=0}.
\end{equation*} 
We have to work a bit to write down precisely all the quantities appearing in the formulas above. 
It is straightforward to compute the matrix representations of the maps $A_0^s$ and $A_1^s$. In this setting the projections onto $\Pi_0$ and $\Pi_0^\perp$ are given by:
\begin{equation*}
	pr_{\Pi_0} = \begin{pmatrix}
		1 &0 \\ 0 &0
	\end{pmatrix}, \quad  pr_{\Pi_0^\perp} = \begin{pmatrix}
	0 &0 \\ 0 &1
\end{pmatrix}
\end{equation*}

Recall that the definition of $A_0^s$ and $A_1^s$ given in \eqref{eq:def_Ai_intro} depends on the choice of two scalar products $g_0$ and $g_1$. Denote by $G_0$ and $G_1$ their restriction to $\Pi_0^\perp$  and $\Pi_1^\perp$ respectively. We have that:
\begin{equation*}
	\begin{split}
		A_0^s = \begin{pmatrix}
			1 & (1-s)G_0\\ 0 &1
		\end{pmatrix}, \quad A_1^s = \begin{pmatrix}
		1 &(1-s)(G_1-\hat R) \\ 0 &1
	\end{pmatrix}\\
 A_1^s\tilde \Phi_* = \begin{pmatrix}
 	1 & s\hat{R} +(1-s)G_1\\ 0 &1
 \end{pmatrix}
	\end{split}
\end{equation*}

The value of $\Phi_t^s$ and its derivative at $s=0$ is given in \Cref{lemma: quantities at s=0}. Here, for the submatrices of $\Phi^0_t$ and $\partial_s \Phi_t^s\vert_{s=0}$, we use the notation defined in the \Cref{lemma: quantities at s=0}. In this case, since $Y_t = \int_0^t R_\tau d\tau$ and $X_t =1$, we obtain:
\begin{equation*}
	\begin{split}
	\Phi_t^0 & = \begin{pmatrix}
	1 &0 \\ \Gamma &1
\end{pmatrix}= \begin{pmatrix}
1 &0 \\
t &1
\end{pmatrix}\\ \partial_s \Phi_t^s\vert_{s=0} &=\begin{pmatrix}
	\Theta &0\\ \Omega &-\Theta^*
\end{pmatrix} = \begin{pmatrix}
	\int_0^t\int_0^\tau R_\omega d\omega & 0 \\ \int_0^t \int_0^\tau \int_r^\tau R_\omega d\omega dr d\tau & -\int_0^t \int_0^\tau R_\omega d\omega
\end{pmatrix}
\end{split}
\end{equation*}
Let us compute the value of $\det(Q^s)$ in zero. Putting all together we have:
\begin{equation*}
		\det(Q^s)|_{s=0} = \det\left(\begin{pmatrix}
			1 &G_1 \\
			0&1 
		\end{pmatrix} \begin{pmatrix}
		1 &0 \\
		\Gamma &1 
	\end{pmatrix} \begin{pmatrix}
	1 &G_0 \\
	0&1 
\end{pmatrix} -\begin{pmatrix}
1 &0 \\
0&1 
\end{pmatrix}  \right)
\end{equation*}
After a little bit of computation we find that $Q_s\vert_{s=0}$ satisfies:
\begin{equation*}
	\begin{split}
	Q^s\vert_{s=0} = \begin{pmatrix}
		G_1 \Gamma &G_0+G_1+G_1\Gamma G_0 \\ \Gamma &\Gamma G_0 
	\end{pmatrix}, \quad \det(Q^s\vert_{s=0} ) = (-1)^n\det(\Gamma)\det(G_1+G_0),\\ 	(Q^s\vert_{s=0} )^{-1} = \begin{pmatrix}
	-G_0(G_1+G_0)^{-1} &\Gamma^{-1}+ G_0(G_1+G_0)^{-1}G_1 \\
	(G_1+G_0)^{-1} &-(G_1+G_0)^{-1}G_1
\end{pmatrix}
\end{split}
\end{equation*}

We can compute the derivative $\det(Q^s)$ at $s=0$, we find that:
\begin{equation*}
	\begin{split}
		\partial_s Q^s\vert_{s=0} &= (\partial_s A_1^s)\tilde \Phi_* \Phi_1^sA_0^s + A_1^s\tilde \Phi_* (\partial_s\Phi_1^s)A_0^s + A_1^s\tilde \Phi_* \Phi_1^s(\partial_sA_0^s)\vert_{s=0} \\
		&=\begin{pmatrix}
		( G_1-\hat R)\Gamma & G_1- \hat R\\ 0 & 0
	\end{pmatrix}\begin{pmatrix}
	1 & G_0\\ 0 &1
\end{pmatrix}+\\&\quad+\begin{pmatrix}
1 & G_1\\ 0 &1
\end{pmatrix}\left(\begin{pmatrix}
\Theta & 0\\ \Omega &-\Theta^*
\end{pmatrix}\begin{pmatrix}
1 & G_0\\ 0 &1
\end{pmatrix}+\begin{pmatrix}
0 & G_0 \\
0 &\Gamma G_0 
\end{pmatrix}\right) 
	\end{split}
\end{equation*}
We use now Jacobi formula for the derivative of the determinant of a family of invertible matrices. It reads:
\begin{equation*}
	\partial_s \det (M_s) = \det(M_s) \tr (\partial_s M_s M_s^{-1}).
\end{equation*}
Without going into the detail of the actual computation, which at this point is just matrix multiplication, we have that:
\begin{equation*}
	\begin{split}
		\partial_s \det(Q^s)\vert_{s=0} &=  \det(Q^s)\vert_{s=0} \,  \tr(\partial_s(Q^s)(Q^s)^{-1})\vert_{s=0} \\&= \det(Q^s)\vert_{s=0} \,  \tr((G_1+G_0)^{-1}(G_0-G_1+\hat R)+\Gamma^{-1}\Omega)
	\end{split}
\end{equation*}
The last quantity we have to compute is $\tr(K)$. To do so we use \Cref{lemma: second variation general BC and trace}. Mind that in the statement of the Lemma one works with twice the variables, taking as state space $\mathbb{R}^{n} \times \mathbb{R}^{n}$ and using the symplectic form $(-\sigma) \oplus\sigma$ on $\mathbb{R}^{2n}\times \mathbb{R}^{2n}$. The quantities with $\tilde \,$ on top always refer to the system in $\mathbb{R}^{4n}$, where we have a trivial dynamic on the first factor and the boundary condition we impose are in this case $\Delta \times \Delta$ (see the beginning of \Cref{section: proof general BC} for more details). The formula given in the lemma reads:
\begin{equation}
	\label{eq:trace_in_applications}
	\begin{split}
		\tr(K) =& -\dim(N) +\tr[\pi_*^1\tilde \Phi_*^{-1}\,pr_1 \tilde J_1\tilde 	\Phi_* (\tilde Z_0) ]\\ &+\tr\left [\Gamma^{-1} \left( \Omega  + (\pi^2_*-\pi_*^1)\tilde \Phi_*^{-1}\,pr_1 \tilde J_1\tilde 	\Phi_* (\int_0^1 \tilde Z_tZ_t^*J\vert_{\Pi}  dt)\right)\right]
	\end{split}
\end{equation}	
Let us explain all the objects appearing in the formula. $\tilde \pi_*^i$ denotes the differential of the natural projection on the $i-$th factor. The matrix $\tilde \Phi_*$ is given here by:
\begin{equation*}
	\tilde \Phi_* = \begin{pmatrix}
		\begin{pmatrix}
			1 & 0\\ 0 &1
		\end{pmatrix} & 0\\
		0 &\begin{pmatrix}
			1 &R \\ 0 &1
		\end{pmatrix}
	\end{pmatrix}
\end{equation*} Moreover the matrices $\tilde Z_0$, $\tilde Z_t$ and $\tilde Z_1$ are:
\begin{equation*}
	\tilde Z_0 =  \begin{pmatrix}
		0 \\ 1 \\ 0 \\ 1
	\end{pmatrix},\quad \tilde Z_t =\begin{pmatrix}
	0 \\ 0 \\ \int_0^t R_\tau d\tau  \\ 1
\end{pmatrix}, \quad  \tilde Z_1  = \begin{pmatrix}
	0 \\ 1 \\ -R \\ 1
\end{pmatrix}, \quad  \tilde \Phi_*\tilde Z_1= \begin{pmatrix}
0 \\ 1 \\ 0 \\ 1
\end{pmatrix}
\end{equation*} 
The map $pr_1$ denotes the orthogonal projection onto the image of $\tilde Z_1$. We are using the scalar product $g_0 \oplus g_1$ on $T_{\lambda_0} T^*M \times T_{\lambda_1} T^*M$ to define it. One can check that the following map is the coordinate representation of $pr_1$:
\begin{align*}
	pr_1 = (G_0+G_1)^{-1}\begin{pmatrix}
		0 & 0 &0 &0\\
		0 &G_0 &0 &G_1\\
	0 & 0 &0 &0\\
0 &G_0 &0 &G_1\\
	\end{pmatrix} \\
 pr_1 \tilde J_1  = (G_0+G_1)^{-1} \begin{pmatrix}
	0 & 0 & 0 & 0\\
	-1 & 0& 1 & 0\\
		0 & 0 & 0 & 0\\
	-1 & 0& 1 & 0\\
\end{pmatrix}
\end{align*}

Now everything reduces to some tedious matrix multiplications. The second term in the right hand side of \eqref{eq:trace_in_applications} simplifies as follows:
\begin{equation*}
 \tr[\pi_*^1\tilde \Phi_*^{-1}\,pr_1 \tilde J_1\tilde 	\Phi_* (\tilde Z_0) ] = \tr(\hat R (G_0+G_1)^{-1}).
\end{equation*}

For the third term notice that $(\pi_*^2-\pi_*^1)(\tilde \Phi_*)^{-1} pr_1 $ is identically zero since $\tilde \Phi_*$ does not change the projection on the horizontal part and we are working with periodic boundary conditions.  It follows we are left with $\tr(\Gamma^{-1} \Omega)$. Summing up we have that:
\begin{equation*}
	\tr(K) = \tr(\Omega \Gamma^{-1}+ \hat{R}(G_0+G_1)^{-1}){-\dim(M).}
\end{equation*}

It is natural to think of $g_0$ and $g_1$ as restrictions of globally defined Riemannian metrics. Doing so, since we are working with periodic boundary conditions, amounts to choose $G_0= G_1$. Hence the result.
\end{proof}
\subsection{System with drift and Hill-type formulas}
In this section we give a version of Hill's formula for linear systems with drift. They are again linear system with quadratic cost of the following form:  
\begin{equation}
\label{eq: functional hill formula drift}
\begin{split}
	f_u(q) &= A_t q +B_t u(t), \quad \begin{cases}
		\dot q_u = f_u(q) =A_t q_u+B_t u(t),\\  
		q(0) = q_0 \in \mathbb{R}^n
	\end{cases}\\  \mathcal{A}(u) &= \frac{1}{2} \int_0^1 |u|^2+\langle R_t q_u(t),q_u(t) \rangle dt.
\end{split}
\end{equation}
Here $A_t$ is $n\times n$ matrix and $B_t$ a $n\times k$ one, both with possibly non-constant (but continuous) coefficients. The maximized Hamiltonian takes the form:
\begin{equation*}
		H(p,q) = \langle p,A_tq\rangle +\frac{1}{2}\left ( \langle B_tB^*_tp,p\rangle+\langle R_t q,q\rangle\right).
\end{equation*}

Denote by $\hat\Phi_t$ the fundamental solution of $\dot{q } =A_t q$ at time $t$. We can lift this map to a symplectomorphism of the cotangent bundle which we denote by $\underline{\Phi}_t$. 
As boundary conditions we take the following affine subspace of $\mathbb{R}^{4n}$: $$N = \Gamma(\hat \Phi+ \hat\Phi\int_0^1 \hat{\Phi}_r^{-1}B_r\tilde{u}(r)dr).$$ Notice that, since only the tangent space matters in our formulas, the translation is irrelevant and it would be the same as if we considered $\Gamma(\hat \Phi)$. An obvious choice of extremal is the point $\lambda_t = (0,0)$, with control $\tilde{u} =0$.

\begin{thm}[Hill's formula with drift]
	Suppose that a critical point of the functional given in \eqref{eq: functional hill formula drift} is fixed and let $\tilde u $ be its optimal control. Let $G_0$ and $G_1$ be our choices of scalar product, $\Gamma := \int_0^1 \hat\Phi_\tau B_\tau B_\tau^*\hat{\Phi}^*_\tau  d\tau$ and $\Psi_t$ the fundamental solution of $\dot\Psi = \vec{H} \Psi$. Moreover, define the following matrix:
	\begin{equation*}
		G = G_0+\hat\Phi_t^* G_1 \hat\Phi_t.
	\end{equation*}
	Let $I+K$ be the second variation at $\tilde{u}$. The following equality holds:
	\begin{equation*}
		\det(I+K) = \frac{(-1)^n e^{-2 \,\tr(G^{-1}G_0)}}{\det(G)\det(\Gamma)} \det(\Psi_t-\underline{\Phi}_t).
	\end{equation*}
\begin{proof}
The proof is completely analogous to the one of \Cref{thm hill formula}. First of all the Hamiltonian we use to re-parametrize is given as follows:
\begin{equation*}	
	h^t_{\tilde{u}(t)} = \langle p,B_t \tilde{u}(t)+A_tq\rangle +\frac{1}{2}\langle R_t q,q\rangle.
\end{equation*}
Hence the flow and its differential are given by:
\begin{equation*}
	\begin{split}
	\tilde \Phi_t(p,q) =& \begin{pmatrix}
		(\hat\Phi_t^*)^{-1} & -(\hat \Phi_t^*)^{-1}\int_0^t \hat \Phi_\tau^*R_\tau\hat \Phi_\tau d\tau \\ 0 &\hat \Phi_t
	\end{pmatrix}\begin{pmatrix}
	p\\ q
\end{pmatrix} \\&+\begin{pmatrix}
	-(\hat{\Phi}_t^*)^{-1}\int_0^t\hat\Phi_\tau^* R_\tau \hat\Phi_\tau\int_0^\tau \hat{\Phi}_r^{-1}B_r\tilde{u}(r)dr d\tau\\ \hat\Phi_t\int_0^t \hat{\Phi}_r^{-1}B_r\tilde{u}(r)dr 
\end{pmatrix}\\
	(\tilde \Phi_t)_* =& \begin{pmatrix}
	(\hat\Phi_t^*)^{-1}_* & -(\hat \Phi_t^*)^{-1}\int_0^t \hat \Phi_\tau^*R_\tau\hat \Phi_\tau d\tau \\ 0 &\hat \Phi_t
\end{pmatrix}
\end{split}
\end{equation*}
 We have to define the map $Q^s$. Similarly as the previous case, we can define a map having as kernel the annihilator to the boundary conditions as follows:
 \begin{equation*}
    (\eta_0,\eta_1) \mapsto  \begin{pmatrix}
    	\hat\Phi^*_t  &0 \\ 0 & (\hat \Phi_t)^{-1} 
    \end{pmatrix} \eta_1-\eta_0 =: T_1 \eta_1 -\eta_0 
 \end{equation*}
Hence:
\begin{equation*}
	\det(Q_s) = \det(T_1 A_1^s\tilde{\Phi}_* \Phi_1^sA_0^s-1).
\end{equation*}
Set as before $\hat{R} = - (\hat \Phi^*)^{-1}\int_0^1 \hat \Phi_\tau^*R_\tau\hat \Phi_\tau d\tau$, the upper right minor of $\tilde \Phi_*$. A quick computation shows that:
\begin{equation*}
	T_1A_1^s\tilde \Phi_* = \begin{pmatrix}
		1  &(1-s)\hat \Phi_t^*G_1\hat \Phi_t+ s \hat \Phi_t^*\hat R \\0 &1
	\end{pmatrix} := \hat A^s_1 \tilde \Phi_*
\end{equation*}

Hence, up to renaming $G_1$ and $\Gamma$ in the proof of \Cref{thm hill formula}, we have:
\begin{equation*}
	\begin{split}
	\det(Q^s)\vert_{s=0} &= (-1)^n \det(\Gamma)\det(\hat \Phi_t^*G_1\hat \Phi_t+G_0) =  (-1)^n \det(\Gamma)\det(G), \\
	\frac{\partial_s \det(Q^s)}{\det(Q^s)}\vert_{s=0}  &= \tr(((\hat \Phi_t^*G_1\hat \Phi_t+G_0)^{-1}(G_0-\hat \Phi_t^*G_1\hat \Phi_t+\hat\Phi_t^*\hat R)+\Gamma^{-1}\Omega)
	\end{split}
\end{equation*}

Now we have to apply \Cref{lemma: second variation general BC and trace} to compute the trace of the compact part of the second variation. Here $pr_1$ and $\tilde Z_1$ are different since we have changed boundary conditions. However we have the same kind of simplification as in the previous case. Let us write explicitly the new objects:
\begin{equation*}
	Z_1 = \begin{pmatrix}
		0 \\ 1 \\ 0 \\\hat \Phi_t
	\end{pmatrix}, \quad pr_1 = \begin{pmatrix}
		0&0&0 &0 \\ 0 &L G_0 &0 &L \hat\Phi_t^* G_1  \\ 0&0&0&0 \\0 &\hat\Phi_tL G_0 &0 &\hat\Phi_t L \hat\Phi_t^* G_1 
	\end{pmatrix}, \quad L =(G_0+\hat\Phi_t^*G_1\hat\Phi_t)^{-1}
\end{equation*} In the end, the trace reads:
\begin{equation*}
	\tr(K) = \tr(\Gamma^{-1}\Omega+\hat \Phi_t^*\hat R(\hat \Phi_t^*G_1\hat \Phi_t+G_0)^{-1}) -\dim(M).
\end{equation*}

Contrary to the previous section the two term do not simplify. We are left with the following equation for $b$:
\begin{align*}
	b &= \dim(M) + \tr((\hat \Phi_t^*G_1\hat \Phi_t+G_0)^{-1}(G_0-\hat \Phi_t^*G_1\hat \Phi_t))\\
	 &= n+ 2\, \tr(G^{-1} G_0)-n = 2\, \tr(G^{-1}G_0).
\end{align*}

Hence, the statement follows evaluating $\det(Q^s)$ at $s =1$:
	\begin{equation*}
		\begin{split}
		\det(Q^s) &= \det(T_1 \tilde \Phi_* \Phi_1^1-1) = \det (\underline{\Phi}_t^{-1}\tilde \Phi_*\Phi_1^1-1)\\
		&= \det(\Psi_t-\underline{\Phi}_t).
	\end{split}
	\end{equation*} 
\end{proof}
\end{thm}

	\section{Proof of the main Theorem}

In this section we provide a proof of \Cref{thm: main theorem}. First we work with separated boundary conditions and then reduce the general case to the former. The proof is a bit long so we try to give here a concise outline. The idea is to construct an analytic function $f$ which vanishes precisely on the set $\{-1/\lambda: \lambda \in \text{Spec}(K)\}\subseteq \mathbb{R}$. Particular care is needed to show that the multiplicity of the zeros of this function equals the multiplicity of the eigenvalues of $K$. We do this in \Cref{prop: boundary value problem determinant} and \Cref{prop: multiplicity root separated} respectively. 
We show that this function decays exponentially and use a classical factorization Theorem by Hadamard to represent it as \begin{equation*}
	f(s) = a s^k e^{b s} \prod_{\lambda \in \text{Spec}(K)}(1+\lambda s)^{m(\lambda)}, \quad a, b\in \mathbb{C}, a\ne 0, k \in \mathbb{N}.
\end{equation*}

To prove the general case, we double the variables and consider general boundary conditions as separated ones. In this framework we compute the value of the parameters $a,b$ and $k$ appearing in the factorization.

\subsection{Separated boundary conditions}
\label{section: proof separate BC}
We briefly recall the notation. We are working with an extremal $\lambda_t$ with initial and final point $(\lambda_0,\lambda_1) \in Ann(N)$, where $N = N_0\times N_1$ are product boundary conditions. We are assuming that $\lambda_t$ is strictly normal and satisfies Legendre strong conditions. We work in a fixed tangent space, namely $T_{\lambda_0}T^*M$, to do so we backtrack our curve to its starting point $\lambda_0$ using the flow generated by the time dependent Hamiltonian:
\begin{equation*}
	h^t_{\tilde{u}}(\lambda) = \langle \lambda, f_{\tilde{u}(t)} (q)\rangle-\varphi_t(q,\tilde{u}(t)) .
\end{equation*}
We denote the differential of said flow by $\tilde \Phi_*$. We have a scalar product $g_i$ on $T_{\lambda_i}T^*M$, for each $i=0,1$. We assume that the orthogonal complement to the fibre at $\lambda_i$, $\Pi_i = T_{\lambda_i}T^*_{\pi(\lambda_i)}M$ is a Lagrangian subspace and that the range of $Z_0$ (and $\tilde \Phi_* Z_1$ respectively) is contained in $\Pi_0^{\perp}$ (resp. $\Pi_1^\perp$).

\begin{rmrk}
If we fix Darboux (i.e. \emph{canonical}) coordinates coming from the splitting $\Pi_i \oplus \Pi_i^\perp$ it is straightforward to check that $g_i$ takes a block diagonal form with symmetric $n \times n$ matrix $G^j_i$ on the main diagonal. Similarly we can write down the coordinate representation of the matrix $J_i$ and find:
\begin{equation*}
	g_i(X,Y) = \left \langle \begin{pmatrix}
		G^1_i &0\\ 0&G^2_i
	\end{pmatrix} \begin{pmatrix}
		X_1 \\X_2
	\end{pmatrix}, \begin{pmatrix}
		Y_1 \\Y_2
	\end{pmatrix} \right \rangle,  \quad J_i = \begin{pmatrix}
		0 &-(G^1_i)^{-1} \\ (G^2_i)^{-1} &0
	\end{pmatrix}.
\end{equation*}
\end{rmrk}
For $s \in \mathbb{R}$ (or $\mathbb{C}$) we introduce the following symplectic maps:
\begin{equation}
	\label{eq: def maps Ai}
	\begin{split}
		A_0^s (\eta) &= \eta + (1-s) J_0^{-1}pr_{\Pi_0^\perp} \eta, \quad  \eta \in T_{\lambda_0}(T^*M), \\
		A_1^s(\eta) &= \eta+ (1-s)(J_1^{-1} +\tilde \Phi_*\circ pr_{\Pi_0}\circ\tilde \Phi_*^{-1})pr_{\Pi_1^\perp} \eta, \quad  \eta \in T_{\lambda_1}(T^*M).
	\end{split}
\end{equation} 	

Notice that the transformation $A_i^s$ are indeed  symplectomorphisms. In (canonical) coordinates given by $\Pi_i$ and $\Pi_i^{\perp}$ they have the following matrix representation:
\begin{equation*}
	A_i^ s = \begin{pmatrix}
		1 &&(1-s)S_i\\ 0 &&1
	\end{pmatrix}, \quad S_i^* = S_i.
\end{equation*} 

The last map we are going introduce is two families of dilation in $T_{\lambda_0}(T^*M)$, one of the vertical subspace and one of its orthogonal complement.  Let $s \in \mathbb{R}$ (or $\mathbb{C}$) and let us define the following maps:
\begin{equation}
	\begin{split}
		\label{eq: def dilation s}
		\delta^s: T_{\lambda_0}T^*M \to  T_{\lambda_0}T^*M, \, \quad \delta^s \nu = s \, pr_{\Pi_0}\nu + pr_{\Pi_0^\perp} \nu \\
		\delta_s: T_{\lambda_0}T^*M \to  T_{\lambda_0}T^*M, \, \quad \delta_s \nu = pr_{\Pi_0}\nu + s \, pr_{\Pi_0^\perp} \nu \\
	\end{split}
\end{equation}

\begin{prop}
	\label{prop: boundary value problem determinant}
	Let $A^s_i$ be the maps given in \cref{eq: def maps Ai} and let $\Phi_1^s$ be the fundamental solution of the system:
	\begin{equation*}
		\dot \eta = Z_t^s(Z_t^s)^* J \eta, \quad Z_t^s = \delta^s Z_t,
	\end{equation*}
	
	The operator $I+sK$ restricted to $\mathcal{V}$ has non trivial kernel if and only if there exists a non zero $(\eta_0,\eta_1) \in T_{(\lambda_0,\lambda_1)} (Ann(N))$ such that
	\begin{equation*}
		A_1^s\circ \tilde{\Phi}_*\circ \Phi_1^s \circ A_0^s \, \eta_0 =\eta_1
	\end{equation*}
	In particular, the geometric multiplicity of the kernel of $I+s K$ equals the number of linearly independent solutions of the above equation. 
\end{prop}

\begin{proof}
	$I+sK$ has a non trivial kernel if and only if $$\langle u,v\rangle + \langle sKu,v\rangle=0, \quad \forall u, v \in \mathcal{V}.$$ 
	Equivalently if and only if $\exists u$ such that  $u+sKu \in \mathcal{V}^{\perp}$ (see \Cref{orthogonal to V} below for a description of $\mathcal{V}^\perp$). This is in turn equivalent to the following system
	\begin{equation}
		\label{eq: eigenvalues}
		\begin{cases}
			 (1-s)u_0 = v_0 \\
			u_t = s Z_t^*J \Big(\int_{0}^t Z_\tau u_\tau d\tau +Z_0u_0\Big ) + Z_t^*J\nu \\
			(1-s)	u_1  = s \Lambda(u)+v_1
		\end{cases}
	\end{equation}

	Let us substitute $Z_t$ with $Z_t^s = \delta^s Z_t$. It is straightforward to check, using the definition in \eqref{eq: def dilation s}, that:
	$$sZ_t^*J \int_{0}^t Z_\tau u_\tau d\tau =  (Z_t^s)^*J \int_{0}^t Z^s_\tau u_\tau d\tau.$$
	 Moreover $ Z_t^*J \nu = (Z_t^s)^* J \nu$ for any $\nu \in \Pi$ and $sZ_t^*JZ_0 = (Z_t^s)^*JZ_0^s$ since $\delta^sZ_0 = Z_0$.

	All the computations we will do from here on are aimed at rewriting \eqref{eq: eigenvalues} as a boundary value problem in $T_{\lambda_0}T^*M\times T_{\lambda_0}T^*M$. Let us start with the second equality in \eqref{eq: eigenvalues}. Define:
	 $$\eta(t) = \int_0^t Z^s_\tau u_\tau d\tau +Z^s_0u_0 +\nu, \quad\eta(0) =  Z^s_0u_0 +\nu. $$
	
	The linear constraint defining $\mathcal{V}$ implies:
	\begin{equation*}
		\begin{split}		
			Z_0u_0 +\int_0^1Z_t u_t dt +Z_1u_1 \in \Pi &\iff Z_0u_0 +\int_0^1Z^s_t u_t dt +Z^s_1u_1 \in \Pi\\
			&\iff \eta(1)+ Z^s_1u_1 \in \Pi_0\\
			&\iff \tilde{\Phi}_*(\eta(1)+ Z^s_1u_1) \in \Pi_1
		\end{split}
	\end{equation*}
	This imposes a condition on the initial and final value of $\eta(t)$. The initial one must lie in $\pi_*^{-1} TN_0$ and final one in $\pi_*^{-1}TN_1$. If we multiply by $Z_t^s$ the second equation in \eqref{eq: eigenvalues}, we are brought to consider the following problem:
	\begin{equation}
		\label{jacobi equation}
		\begin{cases}
			\dot{\eta}(t)= Z_t^s(Z_t^s)^* J \eta(t) \\
			(\pi_*\eta(0),\pi_*\eta(1)) \in T(N_0\times N_1) \\
		\end{cases}
	\end{equation}
	Now we use the remaining equations in \eqref{eq: eigenvalues} to reduce the space $\pi_*^{-1}(T(N_0\times N_1))$ to a Lagrangian one. Let us reformulate the first and third line in \eqref{eq: eigenvalues} as equations in $T_{\lambda_0} T^*M$ and $T_{\lambda_1}T^*M$.  Using the maps $Z_0$ and $\tilde{\Phi}_*Z_1$ we obtain:
	\begin{equation*}
		\begin{cases}
			(1-s) Z_0u_0 = Z_0v_0 = pr_0 J_0 \nu\\
			(1-s) \tilde{\Phi}_*Z_1u_1 = \tilde{\Phi}_*Z_1v_1 + s\tilde{\Phi}_*Z_1\Lambda(u) = pr_1 J_1 \tilde{\Phi}_* \nu +s \tilde{\Phi}_*Z_1\Lambda(u)	\end{cases}
	\end{equation*}
	
	Where we used the fact that $(v_0,v_t,v_1)$ is a vector in $\mathcal{V}^\perp$. Notice that, for $u \in \mathcal{V}$ we have that: $$s(Z_0u_0 +\int_0^1Z_t u_t dt +Z_1u_1) = Z_0u_0 +\int_0^1Z^s_t u_t dt +Z^s_1u_1$$
	This implies that the term $s\tilde{\Phi}_*Z_1\Lambda(u)$ can be rewritten as:
	\begin{equation*}
		\begin{aligned}
			s\tilde{\Phi}_*Z_1\Lambda(u)= s\, pr_1J_1\tilde{\Phi}_*(Z_0u_0 +\int_0^1Z_t u_t dt +Z_1u_1 ) = pr_1 J_1\tilde{\Phi}_*(\eta(1)+Z_1^s u_1-\nu)		
		\end{aligned} 
	\end{equation*}

	If substitute $Z_0u_0$ with $pr_{\Pi_0^\perp} \eta(0)$  we end up with the equations:
	\begin{equation*}
		\begin{cases}
			(1-s)pr_{\Pi_0^\perp} \eta(0) = pr_0 J_0 \nu = pr_0 J_0 \eta(0) \\
			(1-s) \tilde{\Phi}_*Z_1 u_1= pr_1 J_1\tilde{\Phi}_*(\eta(1)+Z_1^su_1)
		\end{cases}
	\end{equation*}
	
	Now we do the same kind of substitution for the term $Z_1^s u_1$. Using the projections on $\Pi_0$ and $\Pi_0^\perp$ and recalling that $\tilde\Phi_*$ sends $\Pi_0$ to $\Pi_1$, we have:
	\begin{equation*}
		\begin{split}
			Z_1^s u_1 &= s \,pr_{\Pi_0} Z_1u_1 + pr_{\Pi_0^\perp} Z_1u_1 = (s-1)\,pr_{\Pi_0} Z_1u_1 + Z_1u_1\\
			pr_1 J_1 \tilde{\Phi}_* Z_1^su_1 &= (s-1)\, pr_1 J_1 \tilde{\Phi}_*pr_{\Pi_0} Z_1u_1 + pr_1 J_1\tilde{\Phi}_*Z_1u_1 \\
			&=    (s-1)\,  pr_1  J_1 \tilde{\Phi}_* pr_{\Pi_0}Z_1u_1
		\end{split}
	\end{equation*}	
	
	Last equality being due to the fact that the image of $\tilde{\Phi}_*Z_1$ is isotropic and thus $J_1\imm(\tilde \Phi_*Z_1) \subset \imm( \tilde{\Phi}_*Z_1)^{\perp}$. Moreover $\tilde \Phi_*Z_1u_1$ coincides with the projection of $-\tilde \Phi_*\eta(1)$ on $\Pi_1^\perp$. Thus we are left with:
	\begin{equation}
		\label{eq: bvp both points}
		\begin{cases}
			(1-s)pr_{\Pi_0^\perp} \eta(0) = pr_0 J_0 \nu = pr_0 J_0 \eta(0) \\
			(1-s)(- pr_{\Pi_1^\perp}\tilde \Phi_* \eta(1) + pr_1  J_1 \tilde{\Phi}_* pr_{\Pi_0}Z_1u_1)= pr_1 J_1\tilde \Phi_*\eta(1)
		\end{cases}
	\end{equation}
	
	It is straightforward to check that $pr_1 J_1 \tilde \Phi_* pr_{\Pi_0} Z_1u_1$ depends only on the projection of $Z_1u_1$ on $\Pi_0^\perp$. Moreover, expanding  $1 = \tilde \Phi_* \circ \tilde \Phi_*^{-1}$ and using the relation $\tilde \Phi_*Z_1u_1=-pr_{\Pi_1^\perp}\tilde \Phi_*\eta(1)$, the second equality in \eqref{eq: bvp both points} can be rewritten as:
	\begin{equation}
		\label{eq: bvp final point}
		(s-1) pr_{\Pi_1^\perp} \tilde \Phi_* \eta(1) = pr_1J_1 \tilde{\Phi}_* \eta(1)+ (1-s) pr_1J_1\tilde \Phi_* pr_{\Pi_0} \tilde \Phi_*^{-1} pr_{\Pi_1^\perp}\tilde \Phi_*  \eta(1)
	\end{equation}
	 
	If $s =1$, the equations reduce to $pr_0(J_0\eta(0)) =0$ and $pr_1 J_1 \tilde \Phi_*\eta(1)=0$. Consider the first case, the relation is equivalent to:
	\begin{equation*}
		\sigma (\eta(0),Z_0w_0) = g_0(J_0\eta(0),Z_0 w_0) = g_0(pr_0 J_0\eta(0),Z_0 w_0) =0, \, \forall w_0 \in \mathbb{R}^{\dim(N_0)}.
	\end{equation*}
	Thus we are looking for solution starting from $T_{\lambda_0} Ann(N_0)$. Similarly setting $s=1$ in the second equality we find that $\eta(1)$ must lie inside $T_{\lambda_1} Ann(N_1)$.

	Now, for $s\ne 1$, we want to interpret the boundary conditions as an analytic family of Lagrangian subspaces depending on $s$. To do so we employ the following linear map defined in \eqref{eq: def maps Ai}:
	\begin{equation*}
		A_0^s (\eta) = \eta + J_0^{-1}(1-s) pr_{\Pi_0^\perp} \eta 
	\end{equation*}
	
	If $\eta \in T_{\lambda_0} Ann(N_0)$ we have that $pr_{\Pi_0^\perp} \eta = pr_0 \eta$  and $pr_0 J_0 \eta=0$ and thus:
	\begin{equation*}
		\begin{aligned}
			pr_{\Pi_0^\perp}(A_0^s(\eta)) &= pr_{\Pi_0^\perp}(\eta)  \in \imm(Z_0) \\
			pr_0 J_0(A_0^s(\eta)) &= pr_0(J_0 \eta+(1-s) pr_{\Pi_0^{\perp}}\eta) \\& =pr_0 J_0\eta +(1-s) pr_{\Pi_0^\perp} \eta \\ 
			&=(1-s) pr_{\Pi_0^\perp}(A_0^s(\eta))
		\end{aligned}
	\end{equation*}
	
	So we have shown that $A_0^s(T_{\lambda_0} Ann(N_0))$ is precisely the space satisfying the first set of equations. A similar argument works for the final point. Let us recall the definition of $A_1^s$ given in \eqref{eq: def maps Ai}:
	\begin{equation*}
		A_1^s(\eta) = \eta+ (1-s)(J_1^{-1}+\tilde \Phi_* pr_{\Pi_0} \tilde \Phi_*^{-1} ) pr_{\Pi_1^\perp} \eta.
	\end{equation*} 	
	
	Now, we check that the boundary condition for the final point are satisfied if and only if $A_1^s \circ \tilde \Phi_* \,\eta(1)\in T_{\lambda_1}Ann(N_1)$. In fact, take any $\eta$ in $T_{\lambda_1}Ann(N_1)$, it holds:
	\begin{equation*}
		\begin{split}
			pr_{\Pi_1^\perp}(A_1^s)^{-1}\eta &= pr_{\Pi_1^\perp}\eta, \\
			pr_{\Pi_1}(A_1^s)^{-1}\eta &=  pr_{\Pi_1}\eta+(s-1)(J^{-1}_1+\tilde \Phi_* pr_{\Pi_0} \tilde \Phi_*^{-1} ) pr_{\Pi_1^\perp} \eta ,\\
			pr_1 J_1 (A_1^s)^{-1} \eta &=(s-1) pr_1 (1+J_1 \tilde \Phi_* pr_{\Pi_0} \tilde \Phi_*^{-1} )pr_{\Pi_1^\perp} \eta. 
		\end{split}
	\end{equation*}
	It is straightforward to substitute the last equality in \eqref{eq: bvp final point} and check that indeed $(A^s_1)^{-1}(T_{\lambda_1} Ann(N_1))$ is the right space. 
	
	Let us call $\Phi_1^s$ the fundamental solution of \eqref{jacobi equation}  at time $1$ and denote by $\Gamma(\Phi_1^s)$ its graph. It follows that $s \in \mathbb{R} \setminus \{0\}$ is in the kernel of $1+s K$ if and only if:
	\begin{equation*}
		\Gamma(\tilde \Phi_* \circ \Phi_1^s) \cap A_0^s(T_{\lambda_0} Ann(N_0))\times (A_1^s)^{-1}(T_{\lambda_1} Ann(N_1)) \ne (0)
	\end{equation*}
	which is equivalent to the condition:
	\begin{equation}
		\label{eq: intersection graph annihilator}
		\Gamma(A_1^s \circ \tilde \Phi_* \circ \Phi_1^s \circ A_0^s) \cap  T_{\lambda_0}Ann(N_0) \times T_{\lambda_1} Ann(N_1)\ne (0)
	\end{equation}
Now we prove the part about the multiplicity. Suppose that two different controls $u$ and $v$ give the same trajectory $\eta_t$ solving \eqref{jacobi equation}. Since the maps $Z_0$ and $Z_1$ are injective it must hold that $v_0 =u_0$ and $v_1 = u_1$. Moreover $\int_0^t Z_\tau u_\tau d\tau = \int_0^t Z_\tau v_\tau d\tau$ $\forall \, t \in [0,1]$ and thus:
\begin{equation*}
	Ku = Z_t^* J\int_0^t Z_\tau u_\tau d\tau = Z_t^* J\int_0^t Z_\tau v_\tau d\tau =  K v.
\end{equation*}
However, Volterra operator are always injective and thus $u=v$.

Vice-versa, consider $u = 0$ and see whether you get solutions of the system above that do not correspond to any variation. Since $u_0$ and $u_1$ are both zero we are considering solution starting from the fibre and reaching the fibre. Plugging in $u_t =0$ we obtain:
\begin{equation*}
	0 = \dot{\eta} =  Z_t^s(Z_t^s)^*J\eta=Z_t^s(Z_t^s)^*J\nu
\end{equation*}

However $pr_{\Pi^\perp}Z_t^s(Z_t^s)J \nu  = X_t X_t^* \nu$ and by assumption the matrix $\int_0^1 X_t X_t^* dt$ is invertible. Thus we get a contradiction.
\end{proof}
	
	The following Lemma was used in the proof of \Cref{prop: boundary value problem determinant}. Gives the orthogonal complement to $\mathcal{V}$ inside $\mathcal{H}$, using the Hilbert structure introduced in \Cref{def:scalar_product_on_H}. We will denote by the symbol $\perp_i$ the orthogonal complement, in $T_{\lambda_i}T^*M$, with respect to the scalar product $g_i$.  	
	\begin{lemma}
		\label{orthogonal to V}
		With this choice of scalar product the orthogonal complement to $\mathcal{V}$ is given by: $$\mathcal{V}^\perp = \{(v_0,Z_t^* J\nu,v_1) : \nu \in \Pi\}$$
		where $v_0$ and $v_1$ are determined by the following conditions:$$  Z_0v_0 - J_0\nu \in \imm Z_0^{\perp_0},\quad  \tilde \Phi_* Z_1v_1 - J_1\tilde \Phi_*\nu \in \imm \tilde \Phi_*Z_1^{\perp_1} .$$
		\begin{proof}
			Suppose that $u \in \mathcal{V}^{\perp}$. Let us test it against infinitesimal variations that fix the endpoints, i.e. such that $u_i = 0$ and $\int_0^1Z_t u_t dt  \in \Pi$. Recall that $\Pi$ is Lagrangian, thus the condition $\int_0^1Z_t u_t dt \in \Pi$ can be equivalently formulated as $\sigma(\int_0^1Z_t u_t dt,\nu) =0$ for all $\nu \in \Pi$. Hence, the subspace $\{u : u_i=0,\int_0^1Z_t u_t dt \in \Pi \}$ is the intersection of the kernels of $u_t \mapsto \int_0^1 \sigma(Z_t u_t,\nu) dt $. It follows that:
			\begin{equation*}
				\langle v, u\rangle  = -\int_0^1 \langle  v_t,u_t\rangle  dt  = 0, \, \forall u \in V \iff v_t = \sigma (\nu,Z_t \cdot ), \quad \nu \in \Pi.
			\end{equation*}
			
			Hence $v_t = Z_t^*J \nu$. 
			Now, take $u \in \mathcal{V}$ and compute:
			\begin{equation*}
				\langle v, u\rangle  =  \sigma (\nu , \int_0^1 Z_tu_t dt) +\sigma (J_0^{-1}Z_0v_0,Z_0u_0)+\sigma (J_1^{-1}\tilde \Phi_*Z_1v_1,\tilde \Phi_*Z_1u_1).
			\end{equation*}
			
			Since $u \in \mathcal{V}$, we have:
			\begin{equation*}
				\sigma \left (\nu , \int_0^1 Z_t u_t dt \right)  =  -\sigma (\nu, Z_0u_0+Z_1u_1).
			\end{equation*}
			
			It follows that: 
			\begin{equation*}
				\begin{split}
					\langle v, u \rangle &= -\sigma (\nu, Z_0u_0+Z_1u_1)+\sigma (J_0^{-1}Z_0v_0,Z_0u_0)+\sigma (J_1^{-1}\tilde \Phi_*Z_1v_1,\tilde \Phi_*Z_1u_1) \\
					& = \sigma(J_0^{-1}Z_0v_0-\nu,Z_0u_0)+\sigma(J_1^{-1}\tilde \Phi_*Z_1v_1-\tilde \Phi_*\nu,\tilde \Phi_*Z_1u_1).
				\end{split}
			\end{equation*}
			
			Since we are assuming that $\int_0^1X_tX_t^*dt >0$, the image of the map $u_t \mapsto \pi_*\int_0^1Z_t u_t dt$ is the whole $T_{q_0}M$. In particular, for any $u_0$, is possible to find variations of the form $(u_0,u_t,0) \in \mathcal 
			V$. An analogous statements holds for variations of the form $(0,u_t,u_1).$

			Hence, if $\langle u, v \rangle = 0$ $\forall u \in \mathcal{V}$, then both 	$\sigma(J_0^{-1}Z_0v_0-\nu,Z_0u_0)$ and $\sigma(J_1^{-1}\tilde \Phi_*Z_1v_1-\tilde \Phi_*\nu,\tilde \Phi_*Z_1u_1)$ must be zero at the same time.
			
			Moreover, this implies that $v_0$ and $v_1$ are completely determined by the value of $\nu$. Finally notice that:
			\begin{equation*}
				\begin{split}
					\sigma(J_0^{-1}Z_0 v_0-\nu,Z_0 u_0) =0 \, \forall\, u_0 &\iff Z_0 v_0-J_0\nu \in \imm(Z_0)^{\perp_0} \\
					\sigma(J_1^{-1}\tilde \Phi_*Z_1 v_1-\tilde \Phi_*\nu,Z_1 u_1)=0 \,\forall\, u_1 &\iff \tilde \Phi_*Z_1 v_1-J_1\tilde \Phi_*\nu \in \imm(\tilde \Phi_*Z_1)^{\perp_1} 
				\end{split}
			\end{equation*} 
		\end{proof}
	\end{lemma}

\begin{rmrk}
	If we complexify all the subspaces involved in the proof of \Cref{prop: boundary value problem determinant}, i.e.  tensor with $\mathbb{C}$ we can take also $s \in \mathbb{C}$.
\end{rmrk}

 We can reformulate the intersection problem in the statement of \Cref{prop: boundary value problem determinant} as follows. Let $ \pi_{N_1}$ the orthogonal projection, with respect to $g_1$, onto the subspace $T_{\lambda_1} Ann(N_1)^{\perp}$ and define a map $Q^s$ as:
\begin{equation}
	\label{eq: def Q separated BC}
	Q^s : T_{\lambda_0}Ann(N_0) \to T_{\lambda_1}Ann(N_1)^{\perp}, \quad Q^s(\eta) = \pi_{N_1}(A_1^s\Phi_1^s A_0^s)(\eta). 
\end{equation}
Let us fix now two bases, one of $T_{\lambda_0} Ann(N_0)$ and one of $T_{\lambda_1}Ann(N_1)$. Construct two $2n \times n$ matrices using the elements of the chosen basis.  Let us call the resulting objects $T_0$ and $T_1$ respectively. It follows that $J_1T_1$ is a basis of $T_{\lambda_1} Ann(N_1)^\perp$.
Define the function $\det(Q^s)$ as the determinant of the $n \times n$ matrix $T_1^*J_1 A_1^s\Phi_1^s A_0^s T_0$. Clearly, different choices of basis give simply a scalar multiple of $\det(Q^s)$ and thus is well defined:
\begin{equation*}
	\det(Q^s) = \frac{\det (T_1^*J_1 A_1^s\Phi_1^s A_0^sT_0)}{\det(T_0^*T_0)^{1/2} \det(T_1^*J_1J_1^*T_1)^{1/2}}
\end{equation*}
 Moreover $\det(Q^s) |_{s=s_0} =0 $ if and only if there exists at least a solution to our boundary problem. Notice that map $s \mapsto \det (Q^s)$ is analytic in $s$ since the fundamental matrix is an entire map in $s$ (see \cite{determinant}[Proposition 4]). The following Proposition shows that the multiplicity of any root $s_0 \ne 0$ is equal to the number of independent solutions to the boundary value problem.
\begin{prop}
	\label{prop: multiplicity root separated}
	The multiplicity of any root $s_0 \ne 0$ of $\det Q^s$ equals the dimension of the kernel of $Q^s$.
	\begin{proof}
		The proof is done in two steps. First we show that the equation $\det(Q^s)=0$ is equivalent to $\det(R^s) =0$ where $R^s$ is a symmetric matrix, analytic in $s$. Once one knows this, it suffices to compute $\partial_s R^s$ and show that it is non degenerate to prove that the multiplicity of the equation is the same as the dimension of the kernel. 
		\proofpart{Replace $Q^s$ with a symmetric matrix} Let us restrict to the case $s \in \mathbb{R}$, since all the roots are real. As shown in \eqref{eq: intersection graph annihilator} and remarked above, the determinant of the matrix $Q^s$ is zero whenever the graph of $A_1^s \tilde \Phi_*\Phi_1^s A_0^s$ intersect the  subspace $L_0 = T_{(\lambda_0,\lambda_1)} (Ann(N_0 \times N_1))$. Suppose that $s_0$ is a time of intersection and choose as coordinates in the Lagrange Grassmannian $L_0$ and another subspace $L_1$ transversal to both $L_0$ and $\Lambda_s :=\Gamma(A_1^s\tilde \Phi_* \Phi_1^s  A_0^s)$. This means that, if $(T_{\lambda_0}T^*M)^2 \approx \{ (p,q)| p,q\in \mathbb{R}^{2n}\}$, we identify $L_0 \approx \{q=0\}$ and $L_1 \approx \{p=0\}$. 
		
		In this coordinates $\Lambda_s$ is given by the graph of a symmetric matrix, i.e. is the following subspace $\Lambda_s = \{(p,R^sp)\}$ where again $R^s$ is analytic in $s$.
		
		The quadratic form associated to the derivative $\partial_s R(s)$ can be interpreted as the velocity of the curve $s \mapsto \Lambda_s$ inside the Grassmannian, it is possible to compute it choosing an arbitrary base of $\Lambda_s$ and an arbitrary set of coordinates. Invariants such as signature and nullity do not change (see for example \cite{beschastnyi_morse,symplecticMethods} or \cite{agrachev_quadratic_paper}). Take a curve $\lambda_s = (p_s,R^sp_s)$ inside $\Lambda_s$ then one has:
		\begin{equation*}
			S(\lambda_s) = \sigma(\lambda_s,\dot{\lambda}_s) = \langle p_s, \partial_s R^s p_s\rangle 
		\end{equation*}
		
		Recall that we will be using the symplectic form given by $(-\sigma_{\lambda_0})\oplus \sigma_{\lambda_1}$, in order to have that graph of a symplectic map is a Lagrangian subspace. 
		
		\proofpart{Replace $\Lambda_s$ with a positive curve}
		We slightly modify our curve to exploit an hidden positivity of the Jacobi equation. We substitute the fundamental solution $\Phi_1^s$ with the following map:
		\begin{equation*}
			\Psi^s = \Psi^s_1= \delta_s \Phi^s_1 \delta_{\frac{1}{s}}.
		\end{equation*}
		
		It is straightforward to check that $\Psi^s$ is again a symplectomorphism and that it is the fundamental solution of the following ODE system at time $t=1$:
		\begin{equation}
			\label{eq: rescaled jacobi eq}
			\dot \Psi_t^s = s Z_tZ_t^*J\Psi_t^s, \quad \Psi_0^s = Id.
		\end{equation}
		
		On one hand we are introducing a singularity at $s=0$ but on the other hand we are going to show that the graph of $\Psi^s$ becomes a monotone curve and its velocity is fairly easy to compute. 
		
		First of all, hoping that the slight abuse of notation will not create any confusion, let us introduce a family of dilations similar to the $\delta_s$,$\delta^s$ also in $T_{\lambda_1}T^*M$. The definition is analogous to the one in \eqref{eq: def dilation s} but with $\Pi_1$ and $\Pi_1^\perp$ instead of $\Pi_0$ and $\Pi_1^\perp$. We will denote them with the same symbol.
		
		Let us consider the following symplectomorphisms:
		\begin{equation*}
		\delta_s A_1^s \tilde \Phi_*\Phi_1^s A_0^s \delta_{\frac{1}{s}} = 	\delta_s A_1^s \tilde \Phi_* \delta_{\frac{1}{s}} \Psi_s \delta_s A_0^s \delta_{\frac{1}{s}}
		\end{equation*}
		Notice that the dilations $\delta_s$ preserve the subspaces $T_{\lambda_i}Ann(N_i)$ and thus the intersection points between the graph of the above map and the subspace $T_{(\lambda_0,\lambda_1)}Ann(N_0\times N_1)$ are unchanged. Let us rewrite the maps $\delta_sA_0^s\delta_{\frac{1}{s}}$ and $\delta_s A_1^s  \tilde \Phi_*\delta_{\frac{1}{s}}$. For the former:
		\begin{equation*}
			\delta_sA_0^s\delta_{\frac{1}{s}} = \delta_s(1+(1-s)J_0^{-1}pr_{\Pi_0^\perp})\delta_{\frac{1}{s}} = 1+ \frac{1-s}{s} J_0^{-1}pr_{\Pi_0^\perp} = B_0^s
		\end{equation*}
		For the latter, a computation in local coordinates and the fact that the dilations $\delta_s$ and $\tilde\Phi_*$ do not commute yield:
		\begin{equation*}
			\begin{split}
				\delta_sA_1^s\tilde \Phi_*\delta_{\frac{1}{s}} &= \delta_s(1+(1-s)(J_1^{-1} +\tilde \Phi_*pr_{\Pi_0}\tilde \Phi_*^{-1})pr_{\Pi_1^\perp})\tilde \Phi_*\delta_{\frac{1}{s}} \\&= (1+ \frac{1-s}{s} J_1^{-1}pr_{\Pi_1^\perp})\tilde \Phi_* = B_1^s\tilde \Phi_*
			\end{split}
		\end{equation*}
		
		Thus we take, for $s\ne 0$, as curve $\Lambda_s := \Gamma(B_1^s \tilde \Phi_* \Psi^s B_0^s)$, the graph of the symplectomorphism just introduced. Notice that $\Psi^s$ is actually analytic, the singularity at $s=0$ comes only form the maps $B_i^s$.

		\proofpart{Computation of the velocity}
		Now we compute the velocity of the graph of $B_1^s \tilde\Phi_* \Psi^s  B_0^s$. Take a curve  $\lambda_s = (\eta, B_1^s\tilde \Phi_* \Psi^s B_0^s\eta)$ inside the $\Lambda_s$ and let us compute the quadratic form associated to the velocity:
		\begin{equation*}
			\begin{split}
				S(\lambda_s) &=-\sigma (\eta,\partial_s \eta)+\sigma(B_1^s \tilde\Phi_* \Psi^s B_0^s\eta,\partial_s( B_1^s \tilde\Phi_* \Psi^s B_0^s\eta)) \\
				&= \sigma(B_1^s \tilde\Phi_* \Psi^s B_0^s\eta,(\partial_s B_1^s ) \tilde\Phi_* \Psi^s B_0^s\eta)+\sigma( \Psi^s  B_0^s\eta,(\partial_s \Psi^s) B_0^s\eta)+\\& \qquad +\sigma(B_0^s\eta,\partial_s(B_0^s)\eta)
			\end{split}
		\end{equation*}
		
		Let us consider the terms of the type $\sigma(B_i^s x,\partial_s B_i^s x)$. It immediate to compute the derivative in this case, recall that $B_i^s x = x + \frac{(1-s)}{s}J_i^{-1} pr_{\Pi_i^\perp} x$. It follows that $\partial_s B_i^s   = -\frac{1}{s^2} J_i^{-1}pr_{\Pi_i^\perp} $ thus the first and last term read as:
		\begin{equation*}
			\begin{split}
				\sigma(B_1^s \xi,(\partial_s B_1^s ) \xi) &= - \frac{1}{s^2}\sigma(\xi, J_1^{-1} pr_{\Pi_1^\perp}\xi) \\ 
				&= \frac{1}{s^2} g_1( pr_{\Pi_1^\perp}\xi, pr_{\Pi_1^\perp}\xi), \quad\text{where } \xi = \tilde \Phi_*\Psi^s B_0^s\eta,\\
				\sigma(\partial_s(B_0^s)\eta, B_0^s\eta)&= - \frac{1}{s^2} \sigma(B_0^s \eta , J_0^{-1} pr_{\Pi_0^\perp} \eta) =  -\frac{1}{s^2} g_0(J_0\eta, J_0^{-1} pr_{\Pi_0^\perp} \eta)\\
					&= \frac{1}{s^2}g_0(pr_{\Pi_0^\perp} \eta,pr_{\Pi_0^\perp} \eta).
			\end{split}
		\end{equation*}
		
		Notice we used the fact that $J_i$ (and thus  $J_i^{-1}$) is $g_i-$skew symmetric. Now we rewrite the middle term. We present it as the integral of its derivative using the equation for $\Psi_t^s$. Let us use the shorthand notation $x = B_0^s\eta$. We obtain:
		\begin{equation*}			
			\begin{split}
				\frac{d}{dt} \left (\sigma(\Psi_t^s x,(\partial_s \Psi_t^s) x)\right) = \sigma(\partial_t\Psi_t^s x,(\partial_s \Psi_t^s) x) + \sigma( \Psi_t^s x,(\partial_s \partial_t\Psi_t^s) x)\\
				= 	s \sigma(Z_tZ_t ^* J \Psi_t^sx,\partial_s\Psi_t^s x)+s \sigma(\Psi_t^s x,Z_tZ_t^* J \partial_s\Psi_t^sx) \\
				+\sigma(\Psi_t^s x,\partial_s\big(s Z_tZ_t^*\big) J \Psi_t^sx).
			\end{split}
		\end{equation*} 
		
		The first and second term have opposite sign and thus cancel out. What remains is:
		\begin{equation*}
			\frac{d}{dt} \left (\sigma((\partial_s \Psi_t^s) x, \Psi_t^s x)\right) =\sigma(\Psi_t^s x, Z_t Z_t^* J \Psi_t^s x)= g(Z_t^*J \Psi_t^s x,Z_t^*J \Psi_t^s x)
		\end{equation*}
		Integrating over $[0,1]$ and using the fact that $\partial_s \Psi_t^s \vert_{(0,0)}= 0$ we get that:
		\begin{equation*}
			\begin{split}
				\sigma(\partial_s\Psi^s x,\Psi^s x) &= \int_0^1 g (Z^*_t J \Psi_t^s x, Z^*_t J \Psi_t^s x) dt.
			\end{split}
		\end{equation*}
		
		Using the notation $\vert \vert \cdot \vert \vert$ to denote the norm with respect to the corresponding metric and summing everything up we find the following expression for the velocity of our curve:
		\begin{equation*}
			S_s(\lambda_s) = \frac{1}{s^2} \left ( \|pr_{\Pi_0^{\perp}}\eta\|^2+\|pr_{\Pi_1^\perp} \Psi^s B_0^s \eta\|^2 \right)+\int_0^1 \|Z_t^*J\Psi_t^s B_0^s \eta \|^2 dt
		\end{equation*}
		
		Since each term of the sum is non negative $S_s(\lambda_s)$ is zero if and only if each term is zero. From the first one we obtain that $\eta$ must be contained in the fibre.		
		Notice that $B^s_0$ acts as the identity on $\Pi_0$ and thus in this case $\|pr_{\Pi_1^\perp} \tilde\Phi_*\Psi^s B_0^s \eta\|^2 = \|pr_{\Pi_1^\perp} \Psi^s \eta\|^2$.
		
		It follows that $\Psi_t^s B^s \eta =\Psi_t^s  \eta  $ is a solution of the Jacobi equation \eqref{jacobi equation} starting and reaching the fibre (recall that $\tilde \Phi_*(\Pi_0) = \Pi_1$). Let us consider now the third piece, since the integrand is positive it must hold that for almost any $t$, $Z_t^*J\Psi_t^s \eta=0$. If we multiply this equation by $Z_t$ we find that:
		\begin{equation*}
			Z_tZ_t^*J \Psi_t^s \eta=0 =\dot\Psi_t^s \eta.
		\end{equation*}
		
		It follow that we are dealing with a constant solution starting and reaching the fibre. However this contradicts the assumption that the matrix $\int_0^1 X_t X_t^* dt$ is non degenerate. In fact, if we substitute a non zero constant solutions starting from the fibre in \eqref{jacobi equation}, we find that $pr_{\Pi_0^\perp} (\eta) = \int_0^1 X_tX_t^* dt\,  \eta \ne 0$
	\end{proof}
\end{prop}

The following proposition is proved in \cite{determinant}.
\begin{prop}
	There exists $c_1,c_2>0$ such that:
	\begin{equation*}
		\|\Phi_1^s\|\le c_1e^{c_2 |s|}\quad  \forall \, s \in \mathbb{C}
	\end{equation*}
	Moreover $\Phi_t^s$ is analytic and the function $s\mapsto \det (Q^s)$ is entire and satisfy the same type of estimate.
\end{prop}

This fact tells us that $\det Q^s$ is an entire function of order $\rho\le 1$. We know its zeros, which are determined by the eigenvalues of $K$, and thus we can apply Hadamard factorization theorem (see \cite{conway}) to present it as an infinite product. It follows that we have the following identity:
\begin{equation}
	\label{eq: hadamard factorization}
	\det(Q^s) = as^k e^{b s}\prod_{\lambda \in Sp(K)} (1+s\lambda)^{m(\lambda)} \quad a, b \in \mathbb{C}, a\ne 0, k \in \mathbb{N}
\end{equation}
where $m(\lambda)$ is the geometric multiplicity of the eigenvalue $\lambda$. To determine the remaining parameters it sufficient to know the value of $\det(Q^s)$ and a certain number of its derivatives at $s =0$ (depending on the value of $k$). Assume for now that $k=0$, a straightforward computation shows that:
\begin{equation}
	\label{eq:normalization_factors}
	\det(Q^s)\vert_{s=0} = a, \quad 	\partial_s\det(Q^s)\vert_{s=0} = a(b+\tr(K)).
\end{equation}
We will compute these quantities in \Cref{lemma: quantities at s=0,lemma: second variation general BC and trace} for general boundary conditions. The proofs can be adapted to the case of separated conditions easily.

	\subsection{General boundary condition}
\label{section: proof general BC}
In this section we prove a determinant formula for general boundary conditions $N \subseteq M \times M$. First, we reduce this case to the case of separate boundary conditions. We have to slightly modify the proof of \Cref{prop: boundary value problem determinant} since, after this reduction,  the Endpoint map will not be a submersion any more.  Then, we compute the normalization factors given in \eqref{eq:normalization_factors}. 

Let us consider $M \times M$ as state space, with the following dynamical system:
\begin{equation}
	\label{eq:doubled_variables}
	f_u(q',q) = \begin{pmatrix}
		0 \\ f_u(q)
	\end{pmatrix}, \quad (q',q) \in M\times M.
\end{equation}
and boundary conditions $\Delta \times N$. With this definition, any extremal between two points $q_0$ and $q_1$ lifts naturally to an extremal between $(q_0,q_0)$ and $(q_0,q_1)$. 
However, the Endpoint map of the new system is no longer surjective. In fact, any trajectory is confined to a submanifold of the form $\{\hat q\} \times M$. Thus, even if we started with a strictly normal extremal, we do not get a strictly normal extremal of the new system. However, there is no real singularity of the Endpoint map here: we have just introduced a certain number of conserved quantities. All the proofs presented above work also in this case. We are going to discuss briefly how to adapt them.

Let us start with Pontryagin maximum principle. It implies that the lift of the extremal curve $\tilde{q} (t) = (q_0,q(t))$ is the curve $\tilde \lambda (t)= (-\lambda_0,\lambda(t))$. This is because the initial and final covector of the lift must annihilate the tangent space of the boundary conditions manifold. In this case $N_0 = \Delta$ and the annihilator of the diagonal subspace is $\{(\lambda,-\lambda): \lambda \in T_{\lambda_0}T^*_{q_0}M\}$. Moreover, by the orthogonality condition in PMP (see \ref{appendix}), we know that $(-\lambda(0),\lambda(1))$ annihilates the tangent space of $N$.

Thus, if we want to work in one fixed tangent space, we have to multiply the first covector by $-1$. This changes the sign of the symplectic form and we are thus brought to work on $T_{\lambda_0} T^*M \times T_{\lambda_0}T^*M$ with symplectic form $(-\sigma)\oplus \sigma$.

With this change of sign, the tangent space to the annihilator of the diagonal gets mapped to the diagonal subspace of $T_{\lambda_0} T^*M \times T_{\lambda_0}T^*M$ and the tangent space to the annihilator of the boundary conditions $N$ is mapped to the tangent space of:
\begin{equation*}
	A(N) = \{(\mu_0,\mu_1) :\langle \mu_0,X_0\rangle =\langle \mu_1,X_1\rangle, \, \forall (X_0,X_1) \in TN\}.
\end{equation*}

Let us make now some notational remarks. We will still denote by $\tilde \Phi_*$ the map  $1 \times \tilde \Phi_*$, which correspond to the new flow we are using to backtrack our trajectory to the starting point. 

As explained in \Cref{section: space of variation}, to define the scalar product on our space of variations, it is necessary to introduce two metrics on the tangent spaces to the endpoints of our curve. We will choose them of the form $\tilde g_0 = g_0\oplus g_0$ and $\tilde g_1 = g_0\oplus g_1$ where $g_0$ and $g_1$ are two metrics on $T_{\lambda_0}T^* M$ and $T_{\lambda_1}T^*M$ respectively.

Now we compute the second variation of the new system \eqref{eq:doubled_variables}. As a general rule, we will denote all the quantities relative to \eqref{eq:doubled_variables} on $M\times M$, putting a $\tilde{}$ on top. We have:
\begin{equation*}
	\tilde Z_0 u_0 =\begin{pmatrix}
		Z_0u_0 \\Z_0u_0
	\end{pmatrix}, \quad \tilde Z_1 = \begin{pmatrix}
		Z_1^0 u_1\\ Z_1^1 u_1
	\end{pmatrix}, \quad \tilde Z_t =\begin{pmatrix}
		0\\Z_t
	\end{pmatrix}.
\end{equation*}
Notice that $\tilde \Phi_* \tilde Z_1$ maps $\mathbb{R}^{\dim(N)}$ to the tangent space to $A(N)$ and we can assume that its image is contained in $\Pi_0^\perp\times\Pi_1^\perp$. We will denote by $\tilde {pr}_1$ the orthogonal projection onto the image of $\tilde Z_1$.

The domain of the second variation is the subspace $\mathcal{V} = \{(u_0,u_t,u_1) : \tilde Z_0u_0 + \int_0^1 \tilde Z_t u_t + \tilde Z_1u_1 \in \Pi_0\times \Pi_0\}$. Clearly, this equation is equivalent to:
\begin{equation*}
	Z_0u_0+Z_1^0u_1 \in \Pi_0 \quad\text{ and } \quad  Z_0u_0+Z^1_1u_1+\int_0^1Z_t u_t dt \in \Pi_0.
\end{equation*}
It follows that the control $u_0$ is completely determined by $u_1$. Moreover we can assume that $Z_0u_0 = -Z_1^0u_1$ since we are free to choose any system of coordinates and any trivialization of the tangent bundle of the manifolds $\Delta$ and $N$.

Let $A_i^s$ be the maps given in \eqref{eq: def maps Ai}.
The following proposition is the counterpart of \Cref{prop: boundary value problem determinant} for general boundary conditions.

\begin{prop}
	\label{prop: jacobi equation general case}
	Let $\Phi_1^s$ be the fundamental solution of the Jacobi system:
	\begin{equation*}
		\dot \eta = Z_t^s(Z_t^s)^* J \eta.
	\end{equation*}
	The operator $1+sK$ restricted to $\mathcal{V}$ has non trivial kernel if and only if there exists a non zero  $(\eta_0,\eta_1) \in T_{(\lambda_0,\lambda_1)}A(N)$ such that
	\begin{equation}
		\label{eq: eigenvalue theorem1}
		A_1^s\circ \tilde \Phi_* \circ  \Phi_1^s \circ  A_0^s \,  \eta_0  =\eta_1
	\end{equation}
	The geometric multiplicity of the kernel equals the number of linearly independent solutions of the above equation. 
	\begin{proof}
		The proof is completely analogous to the one of \Cref{prop: boundary value problem determinant}. However, some slight modifications are in order since the Endpoint map is not surjective in this case. 
		
		\claim{Characterize $\mathcal{V}^\perp$.}		
		The orthogonal complement to $\mathcal{V}$ is given by: \begin{equation*}
			\mathcal{V}^\perp =\{(v_0, \tilde Z_t^*J\tilde \nu,v_1) : \tilde {pr}_1\left(\begin{pmatrix}
				2Z_0v_0 \\0
			\end{pmatrix}+\tilde J_1 \tilde \nu\right) = \tilde \Phi_* Z_1 v_1\}
		\end{equation*}
		The proof is the same as the one of \Cref{orthogonal to V} and yields $v_t = \tilde Z_t^*J \tilde\nu$. Here, however, we can not separate $v_0$ from $v_1$. Take $u \in \mathcal{V}$ and $v\in V^\perp$:
		\begin{align*}
			\langle u, v \rangle = \sigma \left (\int_0^1 Z_t u_t dt,\nu \right)+ \tilde{g}_0(\tilde Z_0 u_0,\tilde Z_0v_0) + \tilde\Phi^*\tilde g_1 (\tilde Z_1 u_1,\tilde Z_1 u_1) \\
			=  \tilde g_0(\tilde Z_0 u_0, \tilde Z_0 v_0-\tilde J_0 \nu)+ \tilde\Phi^*\tilde g_1(\tilde Z_1 u_1, \tilde Z_1 v_1-\tilde J_1 \nu)\\
			= g_0 (Z_1^0 u_1, J_0 \nu-2Z_0v_0+Z_1^0v_1) + \tilde\Phi^*g_1(Z_1^1u_1,Z_1^1 v_1-J_1 \nu)=0.
		\end{align*}
		Hence: 
		\begin{equation}
			\label{eq:orthogonal_complement_nonsep}
			\tilde {pr}_1 \begin{pmatrix}
				2 Z_0 v_0-	J_0 \nu  \\ J_1 \nu
			\end{pmatrix} 
			= \tilde{pr}_1 \left(  \begin{pmatrix}
				2 Z_0 v_0  \\0
			\end{pmatrix} +\tilde 
			J_1 \tilde \nu \right) = \tilde \Phi_* \tilde Z_1 v_1
		\end{equation}
		
		\claim{Derivation of Jacobi equation.}
		Now we can write down Jacobi equation in a fashion similar to the one of \Cref{prop: boundary value problem determinant}. The system reads:
		\begin{equation}
			\label{eq:jacobi_system extended}
			\begin{cases}
				(1-s)\tilde Z_0 u_0 = \tilde Z_0v_0\\
				\dot {\tilde\eta}(t) = \tilde Z_t^s(\tilde Z_t^s) J \tilde \eta(t) \\
				(1-s) \tilde \Phi_* \tilde Z_1 u_1 = \tilde\Phi_* \tilde Z_1 v_1 + s \, \tilde {pr}_1 \tilde J_1 \tilde \Phi_* (\tilde Z_0 u_0+\int_0^1 \tilde Z_t u_t dt +\tilde Z_1 u_1)
			\end{cases}
		\end{equation}
		Where $\tilde\eta(t) = \int_0^t \tilde Z_t^s u_t dt + \tilde Z_0 u_0 +\tilde \nu $. Arguing as in \Cref{prop: boundary value problem determinant}, we can rewrite the last term of the third equation. We have:
		\begin{align*}
			{pr}_1 \tilde J_1 \tilde \Phi_* (\tilde Z_0^s u_0+\int_0^1 \tilde Z^s_t u_t dt +\tilde Z^s_1 u_1) &= {pr}_1 \tilde J_1 \tilde \Phi_* (\eta(1)-\tilde\nu +\tilde Z_1^su_1) \\
			\tilde \Phi_* Z_1^s u_1 = \tilde \Phi_*Z_1^1 u_1 + s\tilde \Phi_* \tilde {pr}_\Pi \tilde Z_1u_1 &= \tilde \Phi_*Z_1^1 u_1 + s(\tilde \Phi_* \tilde {pr}_\Pi\tilde \Phi_* ^{-1}) \tilde \Phi_* \tilde Z_1u_1   .
		\end{align*}
		By the equation defining $\mathcal V$ we have that:
		\begin{equation*}
			\tilde {pr}_{\Pi^\perp }\tilde \Phi_*\eta(1) = - \begin{pmatrix}
				Z_1^0u_1 \\ \tilde \Phi_*Z^1_1u_1
			\end{pmatrix} =- \tilde \Phi_*\tilde Z_1 u_1.
		\end{equation*}
		Lastly note the the first equation gives $(1-s)\tilde Z_0u_0 = (1-s) \tilde {pr}_{\Pi^\perp} \eta(0) = \tilde Z_0 v_0$. Let us now plug all this equations into \eqref{eq:jacobi_system extended} recalling that $\tilde{pr}_{\Pi} + \tilde{pr}_{\Pi^\perp} =1$. The third line now reads:
		\begin{align*}
			pr_1\bigg(\begin{pmatrix}
				2(1-s) \tilde {pr}_{\Pi_0^\perp} \eta(0)  \\ 0
			\end{pmatrix}+\tilde J_1 \tilde \Phi_* (\tilde{pr}_{\Pi}\eta(1)-s \tilde \Phi_* \tilde{pr}_\Pi \tilde \Phi_*^{-1}\tilde {pr}_{\Pi^\perp }\tilde \Phi_*\eta(1))\\ +(1-s)  \tilde{ {pr}}_{\Pi^\perp} \tilde \Phi_* \eta(1)\bigg)=0.
		\end{align*}
		
		Writing the equation component-wise and using the fact that, since $\tilde \Phi_*$ preserve the fibres, it holds:
		$$\tilde \Phi_*pr_{\Pi_0} = \tilde \Phi_* pr_{\Pi_0}\tilde \Phi_*^{-1} \tilde pr_{\Pi_{1}^\perp} \Phi_*   + pr_{\Pi_1} \tilde \Phi_*, $$
		we obtain a relation very similar to the one in \eqref{eq: bvp both points}. Namely:
		\begin{equation}
			\label{eq:proof_general_bc_finaleq}
			\tilde{pr}_1\begin{pmatrix}
				-J_0\left(3(s-1)J_0^{-1} pr_{\Pi_0^\perp} \eta(0)+\tilde \Phi_* pr_{\Pi_0}\eta(0)\right) \\
				J_1 \left(pr_{\Pi_1} \tilde \Phi_*\eta(1) + (1-s)( \tilde \Phi_*  pr_{\Pi_1} \tilde \Phi_*^{-1}+J_1^{-1}) {pr}_{\Pi_1^\perp} \tilde \Phi_* \eta(1)\right)
			\end{pmatrix}=0
		\end{equation}
		At this point the argument is again the following. If $X$ belongs to $A(N) + \Pi$ and $\tilde{pr_1} (\tilde J_1 X) =0$, it follows that:
		\begin{equation*}
			\tilde{g}_1(\tilde pr_1 \tilde J_1 X, \tilde Z_1 u_1) = \sigma(X, \tilde Z_1u_1)=0, \, \forall\,  u_1 \in \mathbb{R}^{\dim(N)} .
		\end{equation*}
		Hence $X$ must lie in  $\mathrm{Im}(\tilde Z_1)^{\angle} \cap (A(N)+\Pi)$, which is $A(N)$. Consider the maps $A^s_i$ defined in \eqref{eq: def maps Ai}, they preserve $A(N) +\Pi$. Thus we can rewrite \eqref{eq:proof_general_bc_finaleq} as:
		\begin{equation*}
			\begin{pmatrix}
				(A_0^s) ^{-1} \eta(0) \\A_1^s\tilde \Phi_* \eta(1)
			\end{pmatrix}
			\in A(N) \iff \Gamma(A_1^s \tilde \Phi_* \Phi_1^s A_0^s)\cap A(N).
		\end{equation*}
		Notice the presence of the inverse of $A_0^s$ due to the sign of the symplectic form. 
		
		\claim{Uniqueness}
		Arguing again as in \Cref{prop: boundary value problem determinant}, to the trivial variation $(0,0,0)$ correspond constant solutions of Jacobi equation starting from the fibre. However, since the Endpoint map of the original system is regular, there are no such solutions. Hence the correspondence between $\ker(1+s K)$ and $\Gamma(A_1^s \tilde \Phi_* A_0^s)\cap A(N)$ is one-to-one.
	\end{proof}
\end{prop}  

Now we define an analogous map to the one in \cref{eq: def Q separated BC}.
Let $\pi_N$ be the orthogonal projection on the space $T_{(\lambda_0,\lambda_1)}A(N)^\perp$ and consider the map:
\begin{equation}
	\label{eq: def map Q}
	Q^s : \Gamma(A_1^s\tilde \Phi_*\Phi_1^s A_0^s) \to T_{(\lambda_0,\lambda_1)}A(N)^\perp, \quad Q^s(\eta) =  \pi_N ( \eta).
\end{equation} 

 Let $T = (T_0,T_1)$ be any linear invertible map from $\mathbb{R}^{2n}$ to the tangent space $T_{(\lambda_0,\lambda_1)}A(N)$. We denote by $J$ the map $(-J_0) \oplus J_1$ representing the symplectic form $(-\sigma_{\lambda_0})\oplus \sigma_{\lambda_1}$. As in the previous section we define the following function:
 \begin{equation*}
 	\det(Q^s)= \frac{\det(T_1^*J_1A_1^s\tilde \Phi_* \Phi_1^sA_0^s -T_0^*J_0)}{\det(T_0^*J_0J_0^*T_0+T_1^*J_1J_1^*T_1)^{1/2}} 
 \end{equation*}

\begin{rmrk}
 One could also define \eqref{eq: def map Q} as a bilinear form, using just the symplectic pairing. In fact, for $(\eta,(\xi_0,\xi_1))$ in $T_{\lambda_0}T^*M\times T_{(\lambda_0,\lambda_1)}A(N)$, define:
 \begin{equation*}
 	\tilde Q^s (\eta,(\xi_0,\xi_1)) = \sigma(A_1^s\tilde \Phi_*\Phi_1^sA_0^s \eta, \xi_1) - \sigma (\eta,\xi_0).
 \end{equation*}
This form is degenerate exactly when $\Gamma(A_1^s \tilde \Phi_* A_0^s)\cap A(N) \ne (0).$
\end{rmrk}

\begin{prop}
	The multiplicity of any roots $s_0 \ne 0$ of the equation $\det(Q^s)$ is equal to the geometric multiplicity of the boundary value problem.
	\begin{proof}
		The same proof of \Cref{prop: multiplicity root separated} works verbatim. Indeed, we are working with the same curve, $\Gamma(A_1^s \tilde \Phi_* A_0^s)$.
	\end{proof}
\end{prop}

In the remaining part of this section we carry out the computation of the normalizing factors of the function $\det(Q^s)$. As already mentioned at the end of the previous section a classical factorization theorem by Hadamard (see \cite{conway}) tells us that:
\begin{equation*}
	\det(Q^s) = as^k e^{b s}\prod_{\lambda \in Sp(K)} (1+s\lambda)^{m(\lambda)} \quad a, b \in \mathbb{C}, a\ne 0, k \in \mathbb{N}
\end{equation*}
where $m(\lambda)$ is the geometric multiplicity of the eigenvalue. We are now going to compute the values of $a,b \in \mathbb{C}$ and $k$.

\begin{thm}
	\label{thm: determinant general BC}
	For almost any choice of metrics $g_0,g_1$ on $T_{\lambda_i}T^*M$, $\det(Q^s\vert_{s=0}) \ne 0$. Whenever this condition holds, the determinant of the second variation is given by:
	\begin{equation}
		\det(1+ sK) =\det((Q^s\vert_{s=0})^{-1})e^{s (\tr(K)-\tr(\partial_s Q^s \,(Q^s)^{-1}\vert_{s=0}))}  \det(Q^s)
	\end{equation}
	\begin{proof}
			We prove the first assertion: for almost any choice of scalar product, $k =0$ and thus $a=\det(Q_1^s|_{s=0})\ne 0$. 
			This is equivalent to a transversality condition between the graph of the symplectomorphism $A_1^s \tilde\Phi_* \Phi^s A_0^s$ and the annihilator of the boundary conditions $N$.
			
			We can argue as follows: consider the following family of maps acting on the Lagrange Grassmannian of $T_{\lambda_0}T^*M \times T_{\lambda_1}T^*M$ depending on the choice of scalar products $G_0$ and $G_1$:
		\begin{equation*}
			F_G = (A_0^s)^{-1} \times  A_1^s\vert_{s=0}, \quad G = (G_0,G_1), G_i>0.
		\end{equation*}
		
		It is straightforward to see that they define a family of algebraic maps of the Grassmannian to itself. For any chosen subspace $L_0$, $F^{-1}_G(L_0)$ is arbitrary close to $\Pi_0 \times \Pi_1$, for $G_i$ large enough. Notice that $\Gamma(A_1^s \tilde\Phi_*\Phi^s A_0^s) \cap L_0 \ne(0)$ if and only if $\Gamma (\tilde \Phi_* \Phi^s)\cap F^{-1}_G(L_0) \ne (0) $. Using the formula in \Cref{lemma: quantities at s=0} one has that $\Gamma(\tilde \Phi_* \Phi^s)$ is transversal to $\Pi_0\times \Pi_1$  and thus to $ F^{-1}_G(L_0)$ for any fixed $L_0$ and $G_i$ sufficiently large. Now, since everything is algebraic in $G$ and there is a Zariski open set in which the transversality condition holds, the possible choices of $G_i$ for which $k>0$ are in codimension 1. 
		
		Let us assume that $k= 0$ and compute $b$. Differentiating the expression for $\det(Q^s)$ in \cref{eq: hadamard factorization} at $s=0$ we find that:
		\begin{equation*}
			\partial_s \det(Q^s)\vert_{s=0} = a(b+\tr(K))
		\end{equation*}
		
		An integral formula for the trace of $K$ is given in \Cref{lemma: second variation general BC and trace}. The derivative of $\det(Q^s)$ can be computed using Jacobi formula:
		\begin{equation*}
			\partial_s \det(Q^s)\vert_{s=0} = a\, \tr(\partial_s Q^s \, (Q^s)^{-1})\vert_{s=0}.
		\end{equation*}
		An explicit expression of the derivatives of the map $Q^s$ can be computed using \Cref{lemma: quantities at s=0}.	
		It follows that $b = \tr(\partial_s Q^s \,(Q^s)^{-1})- \tr (K)$ and we obtain precisely the formula in the statement. 
	\end{proof}
\end{thm}

Before giving the explicit formula for $\tr(K)$ and the derivatives of the fundamental solution to Jacobi equation at $s=0$ we need to make some notational remark and write down a formula for the second variation in the same spirit of \Cref{eq: scnd variation quadratic form} and \cref{eq: second variation quadratic form}. We are working on the state space  $M \times M$ with twice the number of variables of the original system and trivial dynamic on the first factor and separated boundary conditions. The left boundary condition manifold is the diagonal of $M \times M$ and the right one is our starting $N$. We apply the formula in \cref{eq: second variation quadratic form} to this particular system, we denote by $\tilde{Z}_t$ and $\tilde Z_i$ the matrices for the auxiliary problem, in general everything pertaining to it will be marked by a tilde. Identifying $T_{(\lambda_0,\lambda_0)} T^*(M\times M)$ with $T_{\lambda_0}T^*M \times T_{\lambda_0}T^*M$ we have that:
\begin{equation*}
	\tilde Z_0 u_0 =\begin{pmatrix}
		Z_0u_0 \\Z_0u_0
	\end{pmatrix}, \quad \tilde Z_1 = \begin{pmatrix}
		Z_1^0 u_1\\ Z_1^1 u_1
	\end{pmatrix}, \quad \tilde Z_t =\begin{pmatrix}
		0\\Z_t
	\end{pmatrix}.
\end{equation*} 
We still work on the subspace $\mathcal{V} = \{(u_0,u_t,u_1) : \tilde Z_0u_0 + \int_0^1 \tilde Z_t u_t + \tilde Z_1u_1 \in \Pi\}$. However it is clear that the this equation implies that:
\begin{equation*}
	Z_0u_0+Z_1^0u_1 \in \Pi_0 \quad\text{ and } \quad  Z_0u_0+Z^1_1u_1+\int_0^1Z_t u_t dt \in \Pi_0
\end{equation*}
It follows that control $u_0$ is completely determined by $u_1$. Moreover we can assume that $Z_0u_0 = -Z_1^0u_1$ since we are free to choose any system of coordinates and any trivialization of the tangent bundle of the manifolds $\Delta$ and $N$.  Technically we are working with different scalar products on each of the copies of $T_ {\lambda_0} T^*M$. However it is easy to see that on the space $\mathcal{V}$ only the sum of this metrics plays a role. We will denote it $g_0$.
Now we are ready to state the following:

\begin{lemma}
	\label{lemma: second variation general BC and trace}
	The second variation of the extended system, as a quadratic form, can be written as $	\langle (I+K ) u,u \rangle $ where $K$ is the symmetric (on $\mathcal{V}$) compact operator given by:
	\begin{equation*}
		\begin{split}
			-\langle K u, u \rangle  = \int_0^1\int_0^t \sigma(Z_\tau u_\tau,Z_t u_t) d\tau dt -\sigma\left (Z_1^0u_1,\int_0^1 Z_t u_t dt \right )\\+ \sigma \left (\int_0^1Z_tu_t dt -Z_1^0u_1,Z_1^1u_1 \right )+ g_0(Z_1^0u_1,Z_1^0u_1)\\ +(\tilde \Phi^*g_1)(Z_1^1u_1,Z_1^1u_1).
		\end{split}
	\end{equation*}
	Moreover, define the following matrices:
	\begin{equation*}
		\Gamma = \int_0^1 X_t X_t^* dt, \quad \Omega = \int_0^1\int_0^t X_tZ_t^*JZ_\tau X^*_\tau d\tau dt.
	\end{equation*}
	Denote by $pr_1$ the projection onto $T_{(\lambda_0,\lambda_1)} A(N)$ and by $\pi^i_*$ the differential of the natural projections $\pi^i: T^*M \to M$ relative to the $i-$th component,  $i=1,2$. The trace of $K$ has the following expression:
	\begin{equation*}
		\begin{split}
		\tr(K) =& -\dim(N) +\tr[\pi_*^1\tilde \Phi_*^{-1}\,pr_1 \tilde J_1\tilde 	\Phi_* (\tilde Z_0) ]\\ &+\tr\left [\Gamma^{-1} \left( \Omega  + (\pi^2_*-\pi_*^1)\tilde \Phi_*^{-1}\,pr_1 \tilde J_1\tilde 	\Phi_* (\int_0^1 \tilde Z_tZ_t^*J\vert_{\Pi}  dt)\right)\right]
		\end{split}
		\end{equation*}	
	\begin{proof}
			The first part is a straightforward computation combining the expression obtained in \Cref{lemma: compact part second variation} for the second variation with the observation concerning the structure of the maps $Z_0$ and $Z_1$ made before the statement and the choice of the Riemannian metrics.

		Now, notice that the codimension of the space giving fixed endpoints variations $\mathcal{V}$ for the extended system in $\mathcal{H}$ is $2 \dim(M)$. Moreover, it is defined as the kernel of the linear functional:
		\begin{equation*}
			\rho: (u_0,u_t,u_1) \mapsto \tilde \pi_*( \tilde Z_0u_0+\int_0^1\tilde Z_tu_t dt +\tilde Z_1u_1) \in T_{\pi(\lambda_0)}M \times  T_{\pi(\lambda_0)}M.
		\end{equation*}  
		It is straightforward to check that the following subspace is $2\dim(M)-$dimensional and  transversal to $\ker \rho$:
		\begin{equation*}
			\mathcal{V}' = \{(u_0, Z_t^*J \nu, 0): \nu \in \Pi, u_0 \in \mathbb{R}^{\dim M} \}
		\end{equation*}
				
		The trace of $K$ on the whole space splits as a sum of two pieces, the trace of $K\vert_\mathcal{V}$ and the trace of $K\vert_{\mathcal{V}'}$. We can then further simplify and compute separately the trace on $\mathcal{V}' \cap \{u_0=0\}$ and its complement $\mathcal{V}' \cap \{\nu=0\}$.
		We are going to compute the trace of $K$ on the whole space and then the trace of  $K\vert_{\mathcal{V}'}$, determining in this way the value of  $K\vert_{\mathcal{V}}$.
		
		Consider $\mathcal{H} = \mathcal{H}_1\oplus \mathcal{H}_2$. Where \begin{equation*}
			\mathcal{H}_1 = \{u : u = (0,u_t,0)\}, \quad \mathcal{H}_2 = \{u : u = (u_0,0,u_1)\}.
		\end{equation*}
		It is straightforward to check that $\mathcal{H}_1^\perp =\mathcal{H}_2$ for our class of metrics and that $\mathcal{H}_1 \equiv L^{2}([0,1],\mathbb{R}^k)$. Using \cref{eq: second variation quadratic form}, the restrictions of the quadratic form $\hat K(u) =\langle u , K u \rangle$ to each one of the former subspaces read:
		\begin{equation*}
			\hat{K}|_{\mathcal{H}_1} (u) = \int_0^1\int_0^t \sigma (Z_t u_t, Z_\tau u_\tau) dt d\tau, \quad 	\hat{K}|_{\mathcal{H}_2} (u) = \sigma(\tilde Z_0u_0, \tilde Z_1u_1)- \Vert u\Vert_\mathcal{H}^2.
		\end{equation*}
		The trace of the first quadratic form is zero (see \cite{determinant}[Theorem 2]) whereas the trace of the second part is just $-\dim(N)$. Thus we have that:
		\begin{equation*}
			\tr(K\vert_\mathcal{V}) = -\dim(N) - \tr(K\vert_{\mathcal{V}'})
		\end{equation*}
		To compute the last piece we apply $K$ to a control $u \in \mathcal{V}'\cap \{u_0=0\}$ using the explicit expression of the operator given in \Cref{lemma: compact part second variation}. Recall that $pr_1$ is projection onto the image of $\tilde \Phi_* \tilde Z_1$. It follows that:
		\begin{equation*}
			\begin{split}
				K (u) &= \left(0,-Z_t^*J\int_0^t Z_\tau u_\tau d\tau,- L \,pr_1 \tilde J_1\tilde \Phi_*\left(\int_0^1 \tilde Z_tu_t dt \right) \right) \\&= \left (0,-Z_t^*J\int_0^t Z_\tau u_\tau d\tau,-\Lambda(u) \right ).
			\end{split}
		\end{equation*}
	
		Now we write $K(u)$ in coordinates given by the splitting $\mathcal{V} \oplus \mathcal{V}'$. To do so we have first to consider $\rho \circ K(u)$. It is given by:
		\begin{equation}
			\label{eq: rho of K(u)}
			\begin{split}
			-\rho \circ K(u) &= \tilde \pi_*\left( \int_0^1 \tilde Z_tZ_t^*J\int_0^t Z_\tau Z_\tau^*J\nu  +\tilde \Phi_*^{-1}\,pr_1 \tilde J_1\tilde \Phi_* (\int_0^1 \tilde Z_tZ_t^*J \nu dt) \right)		\\
			&= \begin{pmatrix} 0 \\ \Omega \nu\end{pmatrix} +\tilde \pi_*\tilde \Phi_*^{-1}\,pr_1 \tilde J_1\tilde \Phi_* (\int_0^1 \tilde Z_tZ_t^*J \nu dt).
		\end{split}
		\end{equation}
		Set $\Gamma = \int_0^1 X_t^*X_t dt$, it easy to check that, if $u \in \mathcal{V}'$, then $$\rho(u) = \begin{pmatrix}
			X_0 u_0 \\  \Gamma \nu + X_0 u_0
		\end{pmatrix} $$ 
		for an invertible matrix $X_0$ which, without loss of generality, can be taken to be identity.
		It follows that the projection on the first component of $\rho\circ K(u)$ is completely determined by the second term in \cref{eq: rho of K(u)}. Let us call $\pi_*^i$ for $i=1,2$ the projection on the $i-$th component. 
		It follows that an element $(Z_0\hat u_0,Z_t^*J \hat \nu,0) = \hat u \in \mathcal 
		V '$ has the same projection as $K(u)$ if and only if:
		$$\rho \circ (K(u)-\hat u) = 0 \iff 
		\begin{cases}
		 \hat u_0 &= X_0^{-1} \pi^1_*\tilde \Phi_*^{-1}\,pr_1 \tilde J_1\tilde \Phi_* (\int_0^1 \tilde Z_tZ_t^*J \nu dt)\\
		\hat \nu &= \Gamma^{-1} \left( \Omega \nu + (\pi^2_*-\pi^1_*)\tilde \Phi_*^{-1}\,pr_1 \tilde J_1\tilde \Phi_* (\int_0^1 \tilde Z_tZ_t^*J \nu dt)\right)		
	\end{cases}$$
		 In particular the restriction to $\mathcal{V}' 	\cap \{u_0=0\}$ is given by:
		\begin{equation*}
			\nu \mapsto \Gamma^{-1} \left( \Omega  + (\pi^2_*-\pi^1_*)\tilde \Phi_*^{-1}\,pr_1 \tilde J_1\tilde \Phi_* (\int_0^1 \tilde Z_tZ_t^*J  dt)\right)\nu.
		\end{equation*}
		
		A similar strategy applied to $\mathcal{V}' \cap \{\nu=0\}$ tells us that the last contribution for the trace is given by the following map:
		\begin{equation*}
			u_0 \mapsto X_0^{-1}\pi_*^1\tilde \Phi_*^{-1}\,pr_1 \tilde J_1\tilde \Phi_*\tilde Z_0u_0.
		\end{equation*}
		It is worth pointing out that indeed the trace does not depend on $X_0$ and that the vector $\int_0^1\tilde Z_t Z_tJ \nu dt $ is the following:
		\begin{equation*}
			\int_0^1\tilde Z_t Z_tJ \nu dt = \begin{pmatrix}
				\begin{pmatrix}
					0 \\0
				\end{pmatrix}\\
			\begin{pmatrix}
				\Theta \nu \\ \Gamma \nu
			\end{pmatrix}
			\end{pmatrix}
		\end{equation*} 
	In particular if the boundary conditions are separated (i.e. $N = N_0\times N_1$) the part of the trace coming from $\mathcal{V}'\cap \{u_0=0\}$ depends only on the projection onto $T_{\lambda_1}N_1$.
	\end{proof}
\end{lemma}

\begin{lemma}
	\label{lemma: quantities at s=0}
	The flow $\Phi_t^s\vert_{s=0}$ and its derivative $\partial_s \Phi_t^s\vert_{s=0}$ are given by:
	\begin{equation*}
		\Phi_t ^s|_{s=0}= \begin{pmatrix}
			1 &0\\ \int_0^t X_\tau X_\tau^* d\tau &1
		\end{pmatrix}, \quad 	\partial_s\Phi_t ^s\vert_{s=0}= \begin{pmatrix}
			\int_0^t Y_\tau X_\tau^* d\tau &0\\ \int_0^t\int_0^\tau X_\tau Z_\tau ^* J Z_r X_r ^*dr d\tau & -\int_0^tX_\tau Y_\tau^* d\tau
		\end{pmatrix}
	\end{equation*}
	\begin{proof}
		It is straightforward to check that $\Phi_1^s\vert_{s=0}$ and $\partial_s \Phi_1^s\vert_{s=0}$ solve the following Cauchy problems:
		\begin{equation*}
			\begin{cases}
				\dot \Phi_t^0 = \begin{pmatrix}
					0 & 0 \\ X_t X_t^* &0
				\end{pmatrix} \Phi_t^0\\
				\Phi_0^0 = Id
			\end{cases} \quad 	\begin{cases}
				\partial_s \dot \Phi_t^s|_{s=0} = \begin{pmatrix}
					0 & 0 \\ X_t X_t^* &0
				\end{pmatrix} \partial_s \Phi_t^s|_{s=0}+ \begin{pmatrix}
					Y_tX_t^* &0\\ 0 &-X_tY_t^*
				\end{pmatrix}\Phi_t^0\\
				\partial_s \Phi_0^s|_{s=0} = 0.
			\end{cases}
		\end{equation*}
		Solving the ODE one obtains the formula in the statement.
	\end{proof}
\end{lemma}

	\appendix

\section{}
\label{appendix}
In this appendix we collect some information concerning Pontryagin Maximum Principle (PMP) and the differentiation of the endpoint map used and mentioned throughout the text. Everything is fairly standard material in geometric control theory, the reader is referred to \cite{bookcontrol,bookJean,bookSubriemannian} for further details. 
\subsection{Pontryagin Maximum Principle}
Let us introduce a useful family of Hamiltonian functions on $T^*M$. They generate a family of Hamiltonian flows which we use to backtrack admissible trajectories $\gamma$ to their initial point. Moreover, they appear in the formulation of PMP and extend the flow of the fields $f_{u(t)}$ to the cotangent bundle. Set:
\begin{equation*}
	h^t_u: T^*M \to \mathbb R, \quad h^t_u(\lambda) = \langle \lambda,f_u\rangle +\nu \varphi(u,\pi(\lambda)), \quad \nu \le 0.
\end{equation*}

In particular, if $\tilde \gamma$ is and admissible curve, we can build a lift, i.e. a curve $\tilde\lambda$ in $T^*M$ such that $\pi(\tilde\lambda) =  \tilde \gamma$, solving $\dot \lambda = \vec h_u(\lambda)$.  
The following wellknown theorem, Pontryagin Maximum Principle, gives a characterization of critical points of $\mathcal{J}$ (as defined in \eqref{eq: minimaztion problem}), for any set of boundary conditions $N$.

\begin{theorem*}[PMP]
	If a control $\tilde u\in L^\infty([0,1],U)$ is a local minimizer for the functional in \cref{eq: minimaztion problem} there exists a curve $\lambda:[0,1]\to T^*M$, $\nu \in \mathbb{R}$ and an admissible curve $q: [0,1] \to M$ such that for almost all $t\in [0,1]$
	\begin{enumerate}
		\item $\lambda(t)$ is a lift of $q(t)$: 
		$$
		q(t) = \pi (\lambda(t));
		$$
		\item $\lambda(t)$ satisfies the following Hamiltonian system:
		$$
		\frac{d \lambda}{dt} = \vec{h}_{\tilde u(t)}(\lambda);
		$$
		\item the control $\tilde u$ is determined by the maximum condition:
		$$
		h_{\tilde u(t)}(\lambda(t)) = \max_{u\in U} h_u( \lambda(t)), \quad \nu\le0;
		$$
		\item the non-triviality condition holds: $(\lambda(t),\nu)\neq (0,0)$;
		\item transversality condition holds:
		$$(-\lambda(0),\lambda(1)) \in Ann(N).$$
	\end{enumerate}
	We call $q(t)$ an extremal curve (or trajectory) and $\lambda(t)$ an extremal.
\end{theorem*}

There are essentially two possibility for the parameter $\nu$, it can be either $0$ or, after appropriate normalization of $\lambda_t$, $-1$.
The extremals belonging to the first family are called \emph{abnormal} whereas the ones belonging to second \emph{normal}.

\subsection{The Endpoint map and its differentiation}

In this subsection we write down the integral expression for the first and second derivative of the endpoint map. Further details can be found in \cite{bookcontrol}[Section 20.3]. Denote by $\mathcal{U}_{q_0}\subset L^{\infty}([0,1],U)$ be the space of admissible controls at point $q_0$ and define the following map:
\begin{equation*}
	E^t: \mathcal U_{q_0} \to M,\quad u\mapsto \gamma_{u}(t)	
\end{equation*}
It takes the control $u$ and gives the position at time $t$ of the solution starting from $q_0$ of: $$\dot{q} = f_{u(\tau)}(q).$$  We call this map \emph{Endpoint map}. It turns out that $E^t$ is smooth, provided that the fields $f_u(q)$ are smooth too.

For a fixed control $\tilde u $ consider the function $h_{\tilde{u}}(\lambda) := h_{\tilde{u}(t)}(\lambda)$ and define the following non autonomous flow which plays the role of parallel transport in this context:
\begin{equation}
	\label{eq: parallel transport}
	\frac{d}{dt} \tilde{\Phi}_t =  \vec{h}_{\tilde{u}}(\tilde{\Phi}_t) \qquad \tilde{\Phi}_0 = Id
\end{equation}

It has the following properties:
\begin{itemize}
	\item It extends to the cotangent bundle the flow which solves $\dot{q} = f^t_{\tilde{u}}(q)$ on the base. In particular if $\lambda_t$ is an extremal with initial condition $\lambda_0$, $\pi(\tilde{\Phi}_t(\lambda_0)) = q_{\tilde{u}}(t)$ where $q_{\tilde{u}}$ is an extremal trajectory.
	\item $\tilde{\Phi}_t$ preserves the fibre over each $q \in M$. The restriction $\tilde{\Phi}_t:\,  T^*_qM \to T^*_{\tilde{\Phi}_t(q)}M $ is an affine transformation.
\end{itemize}

We suppose now that $\lambda(t)$ is an extremal and $\tilde{u}$ a critical point of the functional $\mathcal{J}$.
We use the symplectomorphism $\tilde{\Phi}_t$ to pull back the whole curve $\lambda(t)$ to the starting point $\lambda_0$. We can express all the first and second order information about the extremal using the following map and its derivatives:
\begin{equation*}
	b_u^t(\lambda) = (h_u^t-h_{\tilde{u}}^t)\circ \tilde{\Phi}_t(\lambda)
\end{equation*}

Notice that:
\begin{itemize}
	\item $b_u^t(\lambda_0)\vert_{u =\tilde{u}(t)} =0 = d_{\lambda_0}\, b_u^t\vert_{u =\tilde{u}(t)}$ by definition.
	\item $\partial_u b_u^t\vert_{u =\tilde{u}(t)} = \partial_u (h_u^t\circ \tilde{\Phi}_t)\vert_{u =\tilde{u}(t)} =0$ since $\lambda(t)$ is an extremal and $\tilde{u}$ the relative control.
\end{itemize}

Thus the first non zero derivatives are the order two ones. We define the following maps:
\begin{equation}
	\label{eq: Z_t and H_t}
	\begin{split}
		Z_t  = \partial_u \vec{b}_u^t(\lambda_0)\vert_{u=\tilde{u}(t)} : \mathbb{R}^k = T_{\tilde{u}(t)}U \to T_{\lambda_0}(T^*M) \\
		H_t = \partial_u^2 b_t(\lambda_0)\vert_{u=\tilde{u}(t)} :  \mathbb{R}^k =T_{\tilde{u}(t)}U \to  T^*_{\tilde{u}(t)}U =\mathbb{R}^k
	\end{split}
\end{equation}

We denote by $\Pi=\ker \pi_*$ the kernel of the differential of the natural projection $\pi: T^*M \to M$.
\begin{prop}[Differential of the endpoint map]
	\label{prop: differential end point map}
	Consider the endpoint map $E^t: \mathcal{U}_{q_0} \to M$. Fix a point $\tilde{u}$ and consider the symplectomorphism $\tilde{\Phi}_t$ and the map $Z_t$ defined above. The differential is the following map:
	\begin{equation*}
		d_{\tilde{u}} E (v_t) =d_{\lambda(t)} \pi \circ d_{\lambda_0}\tilde{\Phi}_t(\int_0^t Z_\tau v_\tau d\tau) \in T_{q_t}M
	\end{equation*}
\end{prop}
In particular, if we identify $T_{\lambda_0}(T^* M)$ with $\mathbb{R}^{2m}$ and write $Z_t = \begin{pmatrix}
	Y_t \\ X_t
\end{pmatrix}$, $\tilde{u}$ is a regular point if and only if $v_t\mapsto \int_0^t X_\tau v_\tau d\tau$ is surjective. Equivalently if the following matrix is invertible:
\begin{equation*}
	\Gamma_t = \int_0^t X_\tau X^*_\tau d\tau \in Mat_{n\times n}(\mathbb{R}), \quad \det(\Gamma_t)\ne 0
\end{equation*}

If $d_{\tilde{u}} E^t$ is surjective then $(E^t)^{-1}(q_t)$ is smooth in a neighbourhood of $\tilde{u}$ and is tangent space is given by:
\begin{equation*}
	\begin{split}
		T_{\tilde{u}}(E^t)^{-1}(q_t) = \{v \in L^\infty([0,1],\mathbb{R}^k) :\, \int_0^t X_\tau v_\tau d \tau =0\} \\
		= \{v \in L^\infty([0,1],\mathbb{R}^k) :\, \int_0^t Z_\tau v_\tau d \tau \in \Pi\}
	\end{split}
\end{equation*}

If the differential of the endpoint map is surjective, the set of admissible control becomes smooth (at least locally) and our minimization problem can be interpreted as a constrained optimization problem. We are looking for critical points of $\mathcal{J}$ on the submanifold $\{u \in \mathcal{U} : E^t(u) = q_1\}$. 
\begin{dfn}
	\label{def: strictly normal extremal}
	We say that a normal extremal $\lambda(t)$ with associated  control $\tilde{u}(t)$  is strictly normal if the differential of the endpoint map at $\tilde{u}$ is surjective. 
\end{dfn}

It makes sense to go on and consider higher order optimality conditions. At critical points is well defined (i.e. independent of coordinates) the Hessian of $\mathcal J$ (or the \emph{second variation}). 
Using chronological calculus (see again \cite{bookcontrol} or \cite{ASZ}) it is possible to write the second variation of $\mathcal{J}$ on $\ker dE^t \subseteq L^{\infty}([0,1],\mathbb{R}^k)$.
\begin{prop}[Second variation]
	\label{prop: second variation}
	Suppose that $(\lambda(t),\tilde{u})$ is a strictly normal extremal, i.e. a critical point of $\mathcal{J}$ for fixed initial and final point. For any $u \in L^{\infty}([0,1],\mathbb{R}^k)$ such that $\int_0^1X_tu_t dt  =0$ the second variation of $\mathcal J$ has the following expression:
	\begin{equation*}
		d^2_{\tilde{u}}\mathcal{J}(u) = -\int_0^1 \langle H_tu_t, u_t\rangle dt - \int_0^1\int_0^t \sigma ( Z_\tau u_\tau ,Z_t u_t ) d\tau dt
	\end{equation*}
	The associated bilinear form is symmetric provided that $u,v$ lie in a subspace that projects to a Lagrangian one via the map $u \mapsto \int_0^1 Z_t u_t dt$.		
	\begin{equation*}
		d^2_{\tilde{u}}\mathcal{J}(u,v) = -\int_0^1\langle H_tu_t, v_t\rangle dt - \int_0^1\int_0^t \sigma ( Z_\tau u_\tau ,Z_t v_t ) d\tau dt
	\end{equation*} 
\end{prop}

Through out the paper we make the assumption, which is customarily called \emph{strong Legendre condition}, that the matrix $H_t$ is strictly negative definite and has uniformly bounded inverse. This guarantees that the term:
\begin{equation*}
	\int_0^1-\langle H_tu_t,v_t\rangle dt
\end{equation*}
is equivalent to the $L^2$ scalar product.

	\bibliographystyle{plain}
	\bibliography{ref}
	\section*{Acknowledgements}
    I would like to thank Prof. Andrei Agrachev for the stimulating discussions during the preparation of this work and Ivan Beschastnyi for helpful comments on preliminary versions of this article. I would also wish to warmly thank the anonymous reviewers for the many helpful suggestions which improved greatly the manuscript.

	 	\section*{Declarations of interest} None
	 \section*{Funding} This research did not receive any specific grant from funding agencies in the public, commercial, or
	 not-for-profit sectors.
	 
\end{document}